\documentclass[10pt]{amsart}
\usepackage{amssymb}
\usepackage{bm}
\usepackage{graphicx}
\usepackage[centertags]{amsmath}
\usepackage{amsfonts}
\usepackage{amsthm}
\usepackage{graphicx}
\usepackage[toc,page]{appendix}
\newtheorem{thm}{Theorem}[section]
\newtheorem{cor}[thm]{Corollary}
\newtheorem{lem}[thm]{Lemma}

\newtheorem{prop}[thm]{Proposition}
\newtheorem{claim}[thm]{Claim}

\newtheorem{fact}[thm]{Fact}
\newtheorem{defn}[thm]{Definition}

\theoremstyle{definition}
\newtheorem{rem}{Remark}[section]

\newcommand{\nn}{\mathbb{N}}
\newcommand{\ee}{\varepsilon}
\newcommand{\con}{\smallfrown}
\newcommand{\meg}{\geqslant}
\newcommand{\mik}{\leqslant}
\newcommand{\ave}{\mathbb{E}}

\newcommand{\hj}{\mathrm{HJ}}
\newcommand{\dcs}{\mathrm{DCS}}
\newcommand{\gr}{\mathrm{GR}}

\begin{document}

\title{Measurable events indexed by words}

\author{Pandelis Dodos, Vassilis Kanellopoulos and Konstantinos Tyros}

\address{Department of Mathematics, University of Athens, Panepistimiopolis 157 84, Athens, Greece}
\email{pdodos@math.uoa.gr}

\address{National Technical University of Athens, Faculty of Applied Sciences,
Department of Mathematics, Zografou Campus, 157 80, Athens, Greece}
\email{bkanel@math.ntua.gr}

\address{Department of Mathematics, University of Toronto, Toronto, Canada M5S 2E4}
\email{ktyros@math.toronto.edu}


\maketitle


\begin{abstract}
For every integer $k\meg 2$ let $[k]^{<\nn}$ be the set of all \textit{words} over $k$. A \textit{Carlson--Simpson tree} of $[k]^{<\nn}$
of dimension $m\meg 1$ is a subset of $[k]^{<\nn}$ of the form
\[ \{w\}\cup \big\{w^{\con}w_0(a_0)^{\con}...^{\con}w_{n}(a_n): n\in \{0,...,m-1\} \text{ and } a_0,...,a_n\in [k]\big\} \]
where $w$ is a word over $k$ and $(w_n)_{n=0}^{m-1}$ is a finite sequence of left variable words over $k$. We study the behavior of a family of
measurable events in a probability space indexed by the elements of a Carlson--Simpson tree of sufficiently large dimension. Specifically we show
the following.
\medskip

\noindent \textit{For every integer $k\meg 2$, every $0<\ee\mik 1$ and every integer $n\meg 1$ there exists a strictly positive constant
$\theta(k,\ee,n)$ with the following property. If $m$ is a given positive integer, then there exists an integer $\mathrm{Cor}(k,\ee,m)$
such that for every Carlson--Simpson tree $T$ of $[k]^{<\nn}$ of dimension at least $\mathrm{Cor}(k,\ee,m)$ and every family $\{A_t:t\in T\}$
of measurable events in a probability space $(\Omega,\Sigma,\mu)$ satisfying $\mu(A_t)\meg \ee$ for every $t\in T$, there exists a 
Carlson--Simpson tree $S$ of dimension $m$ with $S\subseteq T$ and such that for every nonempty $F\subseteq S$ we have}
\[ \mu\Big( \bigcap_{t\in F} A_t\Big) \meg \theta(k,\ee,|F|). \]

The proof is based, among others, on the density version of the Carlson--Simpson Theorem established recently by the authors, as well as,
on a partition result -- of independent interest -- closely related to the work of T. J. Carlson, and H. Furstenberg and Y. Katznelson.
The argument is effective and yields explicit lower bounds for the constants $\theta(k,\ee,n)$. 
\end{abstract}


\section{Introduction}

\numberwithin{equation}{section}

\subsection{Overview}

The present paper -- which is the sequel to \cite{DKT1,DKT2} -- is devoted to the study of the structure of a family of measurable events
in a probability space indexed by a \textit{Ramsey space} \cite{C}. The most classical and illuminating case is when the events are indexed
by the natural numbers. Specifically, let $(\Omega,\Sigma,\mu)$ be a probability space and $\{A_i:i\in\nn\}$ be a family of measurable events
in $(\Omega,\Sigma,\mu)$ satisfying $\mu(A_i)\meg \ee>0$ for every $i\in\nn$. Using Ramsey's Theorem \cite{Ra} and elementary probabilistic
estimates, for every $0<\theta<\ee$ we may select an infinite subset $L$ of $\nn$ such that for every integer $n\meg 1$ and every subset $F$
of $L$ of cardinality $n$ we have
\begin{equation} \label{e11}
\mu\Big( \bigcap_{i\in F} A_i\Big) \meg \theta^n.
\end{equation}
In other words, the events in the family $\{A_i:i\in L\}$ are at least as correlated as if they were independent. 

Now suppose that the events are indexed by another Ramsey space $\mathcal{S}$. A natural problem is to decide whether the aforementioned result
is valid in the new setting. Namely, given a family $\{A_s:s\in\mathcal{S}\}$ of measurable events in a probability space $(\Omega,\Sigma,\mu)$
satisfying $\mu(A_s)\meg\ee>0$ for every $s\in\mathcal{S}$, is it possible to find a ``substructure'' $\mathcal{S}'$ of $\mathcal{S}$ such that
the events in the family $\{A_s:s\in\mathcal{S}'\}$ are highly correlated? And if yes, then one would like to get explicit (and, hopefully, optimal)
lower bounds for their joint probability.

In all cases of interest, this problem is essentially equivalent to that of finding ``copies'' of given configurations inside dense
sets of discrete structures, a theme of fundamental importance in Ramsey Theory. The equivalence between the two perspectives is discussed in
detail in \cite[\S 8.1]{DKT3} and is based on the ``regularity method"\!, a remarkable discovery of E. Szemer\'{e}di \cite{Sz-Reg} asserting
that dense sets of discrete structures are inherently pseudorandom.

\subsection{The main result}

Our goal in this paper is to study the above problem when the events are indexed by words; recall that, for a given integer $k\meg 2$,
a \textit{word over $k$} is just a finite sequence having values in $[k]:=\{1,...,k\}$. If $n\in\nn$, then $[k]^n$ stands for the set
of words over $k$ of length $n$. The set of all words over $k$ is denoted by $[k]^{<\nn}$. 

In this context the most natural and fruitful notion of ``substructure'' is that of a \textit{Carlson--Simpson tree} \cite{CS,DKT3}.
To recall the definition we need, first, to introduce some pieces of notation and some terminology. Specifically, let $k\meg 2$ and fix a letter
$x$ that we regard as a variable. A \textit{variable word over $k$} is a finite sequence having values in $[k]\cup\{x\}$ where the letter $x$
appears at least once, while a \textit{left variable word over $k$} is a variable word over $k$ whose leftmost letter is the variable $x$. 
If $w$ is a variable word and $a\in[k]$, then $w(a)$ is the word over $k$ obtained by substituting all appearances of the letter $x$ in $w$
by $a$. The concatenation of two words $u$ and $v$ over $k$ is denoted by $u^{\con}v$.
\begin{defn} \label{d11}
Let $k\in\nn$ with $k\meg 2$. A \emph{Carlson--Simpson sequence over $k$} is a finite sequence $(w, w_0,..., w_{m-1})$ where $m$ is
a positive integer, $w$ is a word over $k$ and $w_0,...,w_{m-1}$ are left variable words over $k$. A subset of $[k]^{<\nn}$ of the form
\begin{equation} \label{e12}
\{w\}\cup \big\{w^{\con}w_0(a_0)^{\con}...^{\con}w_{n}(a_n): n\in \{0,...,m-1\} \text{ and } a_0,...,a_n\in [k]\big\}
\end{equation}
where $(w,w_0,...,w_{m-1})$ is a Carlson--Simpson sequence over $k$, will be called a \emph{Carlson--Simpson tree} of $[k]^{<\nn}$.
\end{defn}
It is easy to see that the Carlson--Simpson sequence $(w,w_0,...,w_{m-1})$ that generates a Carlson--Simpson tree $T$ via formula
\eqref{e12} is unique. The corresponding positive integer $m$ will be called the \emph{dimension} of $T$ and will be denoted by $\dim(T)$.

We are ready to state the first main result of the paper.
\begin{thm} \label{t12}
For every integer $k\meg 2$, every $0<\ee\mik 1$ and every integer $n\meg 1$ there exists a strictly positive constant $\theta(k,\ee,n)$
with the following property. If $m$ is a given positive integer, then there exists an integer $\mathrm{Cor}(k,\ee,m)$ such that for every
Carlson--Simpson tree $T$ of $[k]^{<\nn}$ of dimension at least $\mathrm{Cor}(k,\ee,m)$ and every family $\{A_t:t\in T\}$ of measurable
events in a probability space $(\Omega,\Sigma,\mu)$ satisfying $\mu(A_t)\meg \ee$ for every $t\in T$, there exists a  Carlson--Simpson
tree $S$ of dimension $m$ with $S\subseteq T$ and such that for every nonempty $F\subseteq S$ we have
\begin{equation} \label{e13}
\mu\Big( \bigcap_{t\in F} A_t\Big) \meg \theta(k,\ee,|F|).
\end{equation}
\end{thm}
Of course, the main point is that, for fixed parameters $k$ and $\ee$, the lower bound on the joint probability of the events $\{A_t:t\in F\}$
given in \eqref{e13} depends only on the cardinality of the set $F$ and not on the dimension of the Carlson--Simpson tree $S$. The argument is
effective and yields explicit estimates for the constants $\theta(k,\ee,n)$. These estimates are admittedly rather weak and it is an important
problem to obtain ``civilized'' bounds. We point out, however, that if we restrict our attention to a certain class of subsets of Carlson--Simpson
trees, then we get optimal lower bounds. This is the content of \S 10 and \S 11 in the main text. One of the consequences of our analysis is
that the constant $\theta(k,\ee,2)$ can be chosen to be $\ee^2-o(1)$, an estimate which is clearly sharp.

We also note that Theorem \ref{t12} implies, naturally, its counterpart for combinatorial subspaces. Recall that, for a pair of positive integers
$n,m$ with $n\meg m$, an \textit{$m$-dimensional combinatorial subspace} of $[k]^n$ is a subset of $[k]^n$ of the form
\begin{equation} \label{e14}
\{ w(a_0,...,a_{m-1}): a_0,...,a_{m-1}\in [k]\}
\end{equation}
where $w$ is an $m$-variable word over $k$ of length $n$ (see \S 2.2 for the precise definition). We have the following corollary.
\begin{cor} \label{c13}
For every $k,m\in\nn$ with $k\meg 2$ and $m\meg 1$, and every $0<\ee\mik 1$ let $\mathrm{Cor}(k,\ee,m)$ be as in Theorem \ref{t12}.
If $N\meg \mathrm{Cor}(k,\ee,m)$ and $\{A_w:w\in [k]^N\}$ is a family of measurable events in a probability space $(\Omega,\Sigma,\mu)$
satisfying $\mu(A_w)\meg \ee$ for every $w\in [k]^N$, then there exists an $m$-dimensional combinatorial subspace $V$ of $[k]^N$
such that for every $n\in \{1,...,|V|\}$ and every subset $F$ of $V$ of cardinality $n$ we have
\begin{equation} \label{e15}
\mu\Big( \bigcap_{w\in F} A_w\Big) \meg \theta(k,\ee,n)
\end{equation}
where $\theta(k,\ee,n)$ is as in Theorem \ref{t12}.
\end{cor}
It is useful to compare Corollary \ref{c13} with the classical density Hales--Jewett Theorem \cite{FK2}. One of its equivalent forms -- see, 
e.g., \cite[Proposition 2.1]{FK2} -- asserts that for every integer $k\meg 2$ and every $0<\ee\mik 1$ there exist a strictly positive constant 
$c(k,\ee)$ and a positive integer $n_0$ such that for every $N\meg n_0$ and every family $\{A_w:w\in [k]^N\}$ of measurable events having all 
probability at least $\ee$, there exists a combinatorial line $L$ of $[k]^N$ (that is, $L$ is a $1$-dimensional combinatorial subspace of
$[k]^N$) such that 
\begin{equation} \label{e16}
\mu\Big( \bigcap_{w\in L} A_w\Big) \meg c(k,\ee). 
\end{equation}
Thus we see that Corollary \ref{c13} extends the density Hales--Jewett Theorem and establishes correlation for the events over an arbitrary 
finite configuration and with a uniform lower bound.

\subsection{On the proof of Theorem \ref{t12}}

The first basic ingredient of the proof of Theorem \ref{t12} is the density version of the Carlson--Simpson Theorem established, recently,
in \cite{DKT3}. The second basic ingredient is a partition result closely related to the work of T. J. Carlson \cite{C}, and H. Furstenberg
and Y. Katznelson \cite{FK1}.

Before we state this result, let us start with a brief motivating discussion. Suppose that we color the set of all pairs of,
say, $[2]^{<\nn}$. Is it then possible to find a Carlson--Simpson tree of $[2]^{<\nn}$ of large dimension all of whose pairs are
of the same color? This natural Ramsey-type problem is easily seen to have a negative answer. Indeed, color red all pairs which are
of the form $\{w(1)^{\con}u(2),w(2)^{\con}u(1)\}$ where $w$ and $u$ are variable words over $2$; color the remaining pairs blue.
Clearly, every Carlson--Simpson tree of $[2]^{<\nn}$ of dimension at least $2$ contains pairs of both colors.

In spite of the existence of pathological colorings, there is non-trivial information on the aforementioned problem. The central
idea -- which has proven to be highly valuable in related parts of Ramsey Theory; see, e.g., \cite{Bl,Ga} -- is to categorize all
pairs of $[2]^{<\nn}$ (and, more generally, all subsets of $[k]^{<\nn}$ of a given cardinality) in a list of classes each of which
has the Ramsey property. This can be done with the help of the notion of type\footnote[1]{We remark that the term ``type" is not used
in \cite{FK1}. Actually, this notion is not named at all either in \cite{FK1}, or in other places in the literature.}, introduced
in \cite[\S 2]{FK1}, which we are about to recall.

Let $k,q\in\nn$ with $k\meg 2$ and $q\meg 1$. Also let $\mathcal{L}$ be a set of cardinality $q$ which is disjoint from $[k]$ and denote by
$\mathrm{W}(k,\mathcal{L})$ the set of all words over $[k]\cup \mathcal{L}$, that is, all finite sequence having values in $[k]\cup \mathcal{L}$.
A \textit{type} of $\mathrm{W}(k,\mathcal{L})$ is a nonempty finite sequence in $\mathcal{L}$ having no consecutive multiple appearances of the
same letter. For instance, if $\mathcal{L}=\{\lambda_1,\lambda_2\}$, then $(\lambda_1,\lambda_2,\lambda_1)$ is a type but $(\lambda_1,\lambda_1,\lambda_2)$
is not. For every $w\in \mathrm{W}(k,\mathcal{L})\setminus [k]^{<\nn}$ we assign its type as follows. First we erase all letters of $w$ which belong
to $[k]$, then we shorten the runs of the same letters of $\mathcal{L}$ to singletons and, finally, we push everything back together. For example,
the type of $(1,\lambda_1, 2, \lambda_1,5, \lambda_2,7, \lambda_2,\lambda_1,8,\lambda_3)$ is the word $(\lambda_1, \lambda_2, \lambda_1, \lambda_3)$. 

Now let $n$ be a positive integer and $\mathbf{w}=(w, w_0,...,w_{n-1})$ be a Carlson--Simpson sequence over $k$. Also let
$v=(a_0,...,a_{m-1})\in \mathrm{W}(k,\mathcal{L})$ be a word of length at most $n$. We set
\begin{equation} \label{e17}
\mathbf{w}(v)=w^\con w_0(a_0)^\con w_1(a_1)^\con...^\con w_{m-1}(a_{m-1})
\end{equation}
with the convention that $\mathbf{w}(v)=w$ if $v$ is the empty word. Finally let
\begin{equation} \label{e18}
\mathbf{w}(k,\mathcal{L})=\big\{\mathbf{w}(v): v \in \mathrm{W}(k,\mathcal{L}) \text{ is of length at most } n\big\}.
\end{equation}
We are ready to state the second main result of the paper (see \S 2.3 for unexplained notation and terminology). 
\begin{thm} \label{t14}
Let $k,q,d,r\in\nn$ with $k\meg 2$ and $q,d,r\meg 1$. Then there exists a positive integer $N$ with the following property. If $n\meg N$ and
$\mathcal{L}$ is a set of cardinality $q$ which is disjoint from $[k]$, then for every Carlson--Simpson sequence $\mathbf{w}=(w,w_0,...,w_{n-1})$
over $k$ and every $r$-coloring of \ $\mathbf{w}(k,\mathcal{L})$ there exists a Carlson--Simpson subsequence $\mathbf{v}=(v,v_0,...,v_{d-1})$
of \ $\mathbf{w}$ such that for every type $\tau$ of \ $\mathrm{W}(k,\mathcal{L})$ the set of all words in $\mathbf{v}(k,\mathcal{L})$ of type
$\tau$ is either empty or monochromatic. The least integer $N$ with this property will be denoted by $\mathrm{FK}(k,q,d,r)$.

Moreover, there exists a primitive recursive function $\varphi:\nn^4\to\nn$ belonging to the class $\mathcal{E}^{6}$ of Grzegorczyk's
hierarchy such that 
\begin{equation} \label{e19}
\mathrm{FK}(k,q,d,r)\mik \varphi(k,q,d,r)
\end{equation}
for every $k,q,d,r\in\nn$ with $k\meg 2$ and $q,d,r\meg 1$.
\end{thm}
Theorem \ref{t14} is a finite ``left-variable word" version of \cite[Theorem 2.7]{FK1} and is closely related to \cite[Lemma 5.9]{C}.
We also note that an infinite extension of Theorem \ref{t14} -- concerning both left and right variable words -- has been obtained by
N. Hindman and R. McCutcheon \cite{HM}. In particular, Theorem \ref{t14} can be derived by \cite[Theorem 2.9]{HM} via a standard compactness
argument but, of course, this reduction is ineffective and gives no information on the numbers $\mathrm{FK}(k,q,d,r)$.

\subsection{Structure of the paper}

The paper is organized as follows. In \S 2 we set up our notation and terminology and we gather some background material needed in the rest of
the paper. In \S 3 we give the proof of Theorem \ref{t14}. As we have already pointed out, Theorem \ref{t14} is the main tool for the analysis
of the Ramsey properties of various families of subsets of Carlson--Simpson trees. This analysis is of independent interest and is carried out
in \S4, \S 5 and \S 6. The reader will find in \S 4.1 a discussion on the content of these sections.

The next three sections are devoted to the proof of Theorem \ref{t12}. The main bulk of the argument is contained in \S 7 and is heavily based
on the material developed in the previous sections. The last step is given in \S 8. In \S 9 we complete the proof of Theorem \ref{t12} and we give
the proof of Corollary \ref{c13}. Finally, in \S 10 and \S 11 we discuss quantitative refinements of Theorem \ref{t12}.


\section{Background material}

\numberwithin{equation}{section}

By $\nn=\{0,1,2,...\}$ we shall denote the natural numbers. For every integer $n\meg 1$ we set $[n]=\{1,...,n\}$. If $X$ is a nonempty finite set,
then by $\ave_{x\in X}$ we shall denote the average $\frac{1}{|X|} \sum_{x\in X}$ where $|X|$ stands for the cardinality of $X$. For every function
$f:\nn\to\nn$ and every $\ell\in\nn$ by $f^{(\ell)}:\nn\to\nn$ we shall denote the $\ell$-th iteration of $f$ defined recursively by $f^{(0)}(n)=n$
and $f^{(\ell+1)}(n)=f\big(f^{(\ell)}(n)\big)$ for every $n\in\nn$.

\subsection{Words over a finite alphabet} 

Let $A$ be a \textit{finite alphabet}, i.e., a nonempty finite set. For every $n\in\nn$ let $A^n$ be the set of all sequences of length $n$
having values in $A$. Precisely, $A^0$ contains just the empty sequence while if $n\meg 1$, then
\begin{equation} \label{e21}
A^n=\big\{ (a_0,...,a_{n-1}): a_i\in A \text{ for every } i\in\{0,...,n-1\}\big\}.
\end{equation}
Also let
\begin{equation} \label{e22}
A^{<n+1}=\bigcup_{i=0}^n A^i \ \text{ and } \ A^{<\nn}=\bigcup_{n\in\nn} A^n.
\end{equation}
The elements of $A^{<\nn}$ are called \textit{words over A}, or simply \textit{words} if $A$ is understood. The \textit{length} of a word $w$ over
$A$, denoted by $|w|$, is the unique natural number $n$ such that $w\in A^n$. For every $i\in\nn$ with $i\mik |w|$ by $w|i$ we shall denote the 
word of length $i$ which is an initial segment of $w$. More generally, if $W$ is a nonempty subset of $A^{<\nn}$ such that for every $w\in W$ 
we have $|w|\meg i$, then we set
\begin{equation} \label{e23}
W|i=\{w|i:w\in W\}.
\end{equation}
The concatenation of two words $w_1$ and $w_2$ over $A$ will be denoted by $w_1^{\con}w_2$. Moreover, for every pair $W_1$ and $W_2$ of nonempty
subsets of $A^{<\nn}$ we set
\begin{equation} \label{e24}
W_1^{\con} W_2=\{w_1^\con w_2: w_1\in W_1 \text{ and } w_2\in W_2\}.
\end{equation}
If $w_1$ and $w_2$ are two words over $A$, then their \textit{infimum}, denoted by $w_1\wedge w_2$, is the greatest common initial segment of
$w_1$ and $w_2$. The infimum operation can be naturally extended to nonempty subsets of words. Specifically, let $W$ be a nonempty subset of 
$A^{<\nn}$. The \textit{infimum} of $W$, denoted by $\wedge W$, is the word over $A$ of greatest length which is an initial segment of every
$w\in W$. Notice that $w_1\wedge w_2=\wedge \{w_1,w_2\}$ for every $w_1,w_2\in A^{<\nn}$.

Finally let $k\in\nn$ with $k\meg 2$. A special case -- which is, nevertheless, of particular importance -- of an alphabet of cardinality $k$
is the set $[k]$. The elements of $[k]^{<\nn}$ will be called \textit{words over $k$}. Of course, if $A$ is another alphabet with $|A|=k$,
then the sets $A^{<\nn}$ and $[k]^{<\nn}$ are naturally ``isomorphic". We need to consider words over arbitrary alphabets for reasons that will
become clear in the sequel.

\subsection{Variable words} 

Let $A$ be a finite alphabet and $n$ be a positive integer. Fix a set $\{x_0,...,x_{n-1}\}$ of distinct letters which is disjoint from $A$.
We view the set $\{x_0,...,x_{n-1}\}$ as a set of variables. An \textit{$n$-variable word over $A$} is a finite sequence having values in 
$A\cup\{x_0,...,x_{n-1}\}$ such that: (a) for every $i\in \{0,...,n-1\}$ the letter $x_i$ appears at least once, and (b) if $n\meg 2$, then
for every $i,j\in\{0,...,n-1\}$ with $i<j$ all occurrences of $x_i$ precede all occurrences of $x_j$. If $A$ is understood, then $n$-variable
words over $A$ will be referred to simply as \textit{$n$-variable words} while $1$-variable words over $A$ will be referred to as 
\textit{variable words}. A \textit{left variable word} is a variable word whose leftmost letter is the variable $x$.
\begin{rem} \label{r21}
The concept of an $n$-variable word over $A$ is closely related to the notion of an $n$-parameter word over $A$ introduced by R. L. Graham
and B. L. Rothschild in \cite{GR}. Indeed, recall that an $n$-parameter word over $A$ is also a finite sequence having values in
$A\cup\{x_0,...,x_{n-1}\}$ satisfying condition (a) above and such that: (b$'$) if $n\meg 2$, then for every $i,j\in\{0,...,n-1\}$
with $i<j$ the first occurrence of $x_i$ precedes the first occurrence of $x_j$. In particular, every $n$-variable word is an $n$-parameter word.
Of course, when ``$n=1$" the two notions coincide.
\end{rem}

As above, let $A$ be a finite alphabet. For every $n$-variable word $w$ over $A$ and every $\beta_0,...,\beta_{n-1}\in A\cup\{x_0,...,x_{n-1}\}$
by $w(\beta_0,...,\beta_{n-1})$ we shall denote the unique word over $[k]\cup\{x_0,...,x_{n-1}\}$ obtained by substituting in $w$ all appearances
of the letter $x_i$ with $\beta_i$ for every $i\in \{0,...,n-1\}$.  Notice that $(\beta_0,...,\beta_{n-1})\in A^n$ if and only if 
$w(\beta_0,...,\beta_{n-1})$ is a word over $A$. More generally, if $m\in [n]$, then $(\beta_0,...,\beta_{n-1})$ is an $m$-variable word if and
only if $w(\beta_0,...,\beta_{n-1})$ is an $m$-variable word. An $m$-variable word of the form $w(\beta_0,...,\beta_{n-1})$ will be called an
$m$-variable \textit{subword} of $w$. 

We recall some basic combinatorial results concerning words. The first one is due to A. H. Hales and R. I. Jewett \cite{HJ}.
\begin{thm} \label{t21}
For every $k,r\in\nn$ with $k\meg 2$ and $r\meg 1$ there exists a positive integer $N$ with the following property. If $n\meg N$, then for every
alphabet $A$ with $|A|=k$ and every $r$-coloring of $A^n$ there exists a variable word $w$ of length $n$ such that the set $\{w(a):a \in A\}$ is
monochromatic. The least integer $N$ with this property will be denoted by $\hj(k,r)$.
\end{thm}
The Hales--Jewett Theorem is one of the cornerstones of modern Ramsey Theory. The best known upper bounds for the numbers $\hj(k,r)$ are due
to S. Shelah. Specifically, by \cite[Theorem 1.5]{Sh}, there is a primitive recursive function $\phi:\nn^2\to\nn$ belonging to the class
$\mathcal{E}^{5}$ of Grzegorczyk's hierarchy such that for every integer $k\meg 2$ and every integer $r\meg 1$ we have that $\hj(k,r)\mik\phi(k,r)$.

We will also need the following theorem.
\begin{thm}\label{t22}
Let $k,d,m,r$ be positive integers with $k\meg 2$ and $d\meg m$. Then there exists a positive integer $N$ with following property.
If $n\meg N$, then for every alphabet $A$ with $|A|=k$ and every $r$-coloring of the set of all $m$-variable words over $A$ of length $n$,
there exists a $d$-variable word $w$ over $A$ of length $n$ such that all $m$-variable subwords of $w$ are monochromatic. The least integer
$N$ with this property will be denoted by $\mathrm{GR}(k,d,m,r)$.

Moreover, there exists a primitive recursive function $\psi:\nn^4\to\nn$ belonging to the class $\mathcal{E}^6$ of Grzegorczyk's hierarchy such that
\begin{equation} \label{e25}
\gr(k,d,m,r)\mik \psi(k,d,m,r)
\end{equation}
for every integer $k\meg 2$, every pair of integers $d\meg m\meg 1$ and every integer $r\meg 1$.
\end{thm}
Theorem \ref{t22} is a variant of the Graham--Rothschild Theorem \cite{GR} which refers to $m$-parameter words instead of $m$-variable words.
We notice that there are several detailed expositions as well as infinite extensions of Theorem \ref{t22} found in the literature -- see, e.g.,
\cite{BBH,FK1,McCbook}. The upper bounds for the numbers $\mathrm{GR}(k,d,m,r)$ mentioned in \eqref{e25} follow from standard arguments and 
the aforementioned work of S. Shelah on the ``Hales--Jewett numbers".

\subsection{Carlson--Simpson sequences}

Let $k\in\nn$ with $k\meg 2$. As we have already mentioned in Definition \ref{d11}, a Carlson--Simpson sequence over $k$ is a nonempty
finite sequence of the form
\begin{equation} \label{e26}
\mathbf{w}=(w, w_0,..., w_{n-1})
\end{equation}
where $n$ is a positive integer, $w$ is a word over $k$ and $w_0,...,w_{n-1}$ are left variable words over $k$.
Of course, Carlson--Simpson sequences over $k$ are in one-to-one correspondence with Carlson--Simpson trees of $[k]^{<\nn}$. In several cases,
however, it is very convenient to work with Carlson--Simpson sequences, and as such, we find it appropriate to explicitly isolate this concept.

Let $\mathbf{w}=(w, w_0,..., w_{n-1})$ be a Carlson--Simpson sequence over $k$. The integer $n$ will be called the \textit{dimension} of
$\mathbf{w}$ and will be denoted by $\dim(\mathbf{w})$. Also let $m\in [n]$. A Carlson--Simpson sequence $\mathbf{v}=(v, v_0,..., v_{m-1})$
over $k$ will be called an \textit{$m$-dimensional Carlson--Simpson subsequence} of $\mathbf{w}$ if there exist a sequence $(a_i)_{i=0}^{n-1}$
in $[k]\cup\{x\}$ and a strictly increasing sequence $(n_i)_{i=0}^{m}$ in $\{0,...,n\}$ such that the following conditions are satisfied.
\begin{enumerate}
\item[(C1)] For every $i\in\{0,...,m-1\}$ we have $a_{n_i}=x$.
\item[(C2)] If $n_0=0$, then $v=w$. Otherwise, we have that $a_0,...,a_{n_0-1}\in [k]$ and
\begin{equation} \label{e27}
v=w^\con w_0(a_0)^\con ... ^\con w_{n_0-1}(a_{n_0-1}).
\end{equation} 
\item[(C3)] For every $i\in\{0,...,m-1\}$ we have
\begin{equation} \label{e28}
v_i(x)=w_{n_i}(a_{n_i})^\con w_{n_i+1}(a_{n_i+1})^\con ...^\con w_{n_{i+1}-1}(a_{n_{i+1}-1}).
\end{equation}
\end{enumerate}
The set of all $m$-dimensional Carlson--Simpson subsequences of $\mathbf{w}$ will be denoted by $\mathrm{Subseq}_m(\mathbf{w})$.

We will need the following theorem. It is a reformulation of \cite[Theorem 4.1]{DKT3}.
\begin{thm} \label{t23}
Let  $k, d, m, r$  be positive integers with $k\meg 2$ and $d\meg m$. Then there exists a positive integer $N$ with the following
property. If $n\meg N$, then for every $n$-dimensional Carlson--Simpson sequence $\mathbf{w}$ over $k$ and every $r$-coloring of 
$\mathrm{Subseq}_m(\mathbf{w})$ there exists $\mathbf{v}\in\mathrm{Subseq}_d(\mathbf{w})$  such that the set $\mathrm{Subseq}_m(\mathbf{v})$
is monochromatic. The least integer $N$ with this property will be denoted by $\mathrm{CS}(k,d,m,r)$.
\end{thm}
The main observation behind the proof of Theorem \ref{t23} is that one can get upper bounds for the numbers $\mathrm{CS}(k,d,m,r)$
which are expressed in terms of the numbers $\mathrm{GR}(k,d,m,r)$. Specifically we have  
\begin{equation} \label{e29}
\mathrm{CS}(k,d,m,r)\mik \mathrm{GR}(k,d+1,m+1,r)
\end{equation}
for every $k,d,m,r\in\nn$ with $k\meg 2$, $d\meg m\meg 1$ and $r\meg 1$. Thus, combining \eqref{e29} and Theorem \ref{t22}, we obtain the 
following corollary.
\begin{cor} \label{c24}
The numbers $\mathrm{CS}(k,d,m,r)$ are bounded by a primitive recursive function belonging to the class $\mathcal{E}^6$
of Grzegorczyk's hierarchy.
\end{cor}

\subsection{Carlson--Simpson trees}

Recall that a Carlson--Simpson tree of $[k]^{<\nn}$ is a set of the form
\begin{equation} \label{e210}
\{w\} \cup \big\{ w^{\con}w_0(a_0)^{\con}...^{\con}w_n(a_n): n\in\{0,...,m-1\} \text{ and } a_0,...,a_n\in [k]\big\}
\end{equation}
where $(w,w_0,...,w_{m-1})$ is a Carlson--Simpson sequence over $k$. Observe that the Carlson--Simpson sequence $(w,w_0,...,w_{m-1})$
that generates a Carlson--Simpson tree $W$ via formula \eqref{e210} is unique. It will be called the \textit{generating sequence} of $W$.
Also recall that the corresponding natural number $m$ is called the \emph{dimension} of $W$ and is denoted by $\dim(W)$. The
$1$-dimensional Carlson--Simpson trees will be called \textit{Carlson--Simpson lines}.

Let $W$ be an $m$-dimensional Carlson--Simpson tree of $[k]^{<\nn}$ and $(w,w_0,...,w_{m-1})$ be its generating sequence.
For every $n\in [m]$ the \textit{$n$-level} $W(n)$ of $W$ is defined by
\begin{equation} \label{e211}
W(n)=\big\{ w^{\con}w_0(a_0)^{\con}...^{\con}w_{n-1}(a_{n-1}): a_0,...,a_{n-1}\in [k]\big\}.
\end{equation}
The \textit{$0$-level} $W(0)$ of $W$ is defined to be the singleton $\{w\}$.

For every $m$-dimensional Carlson--Simpson tree $W$ of $[k]^{<\nn}$ and every $\ell\in [m]$ by $\mathrm{Subtr}_{\ell}(W)$ we shall
denote the set of all $\ell$-dimensional Carlson--Simpson trees of $[k]^{<\nn}$ which are contained in $W$. An element of
$\mathrm{Subtr}_{\ell}(W)$ will be called an \textit{$\ell$-dimensional Carlson--Simpson subtree} of $W$. The set $\mathrm{Subtr}_{\ell}(W)$
is in one-to-one correspondence with the set $\mathrm{Subseq}_{\ell}(\mathbf{w})$ where $\mathbf{w}$ stands for the generating sequence
of $W$. Indeed, notice that for every Carlson--Simpson tree $V$ of $[k]^{<\nn}$ generated by the sequence $\mathbf{v}$ we have that
$V\in \mathrm{Subtr}_{\ell}(W)$ if and only if $\mathbf{v}\in \mathrm{Subseq}_{\ell}(\mathbf{w})$.

A natural example of a Carlson--Simpson tree of $[k]^{<\nn}$ of dimension $m$ is the set $[k]^{<m+1}$. Actually, every Carlson--Simpson
tree of dimension $m$ can be thought of as a ``copy'' of $[k]^{<m+1}$ inside $[k]^{<\nn}$. Specifically, let $W$ be an $m$-dimensional
Carlson--Simpson tree of $[k]^{<\nn}$ and $(w,w_0,...,w_{m-1})$ be its generating sequence. The \textit{canonical isomorphism} associated
to $W$ is the bijection $\mathrm{I}_W:[k]^{<m+1}\to W$ defined by $\mathrm{I}_W(\varnothing)=w$ and
\begin{equation} \label{e212}
\mathrm{I}_W\big((a_0,...,a_{n-1})\big)= w^{\con}w_0(a_0)^{\con}...^{\con}w_{n-1}(a_{n-1})
\end{equation}
for every $n\in [m]$ and every $(a_0,...,a_{n-1})\in [k]^n$. The canonical isomorphism $\mathrm{I}_W$ preserves all structural
properties one is interested in while working in the category of Carlson--Simpson trees. Precisely, we have the following.
\begin{fact} \label{f25}
Let $W$ be an $m$-dimensional Carlson--Simpson tree of $[k]^{<\nn}$. Then the following are satisfied.
\begin{enumerate}
\item[(a)] If $\ell\in [m]$ and $V$ is a Carlson--Simpson subtree of $[k]^{<m+1}$ of dimension $\ell$, then its image $\mathrm{I}_W(V)$
under the canonical isomorphism is an $\ell$-dimensional Carlson--Simpson subtree of $W$.
\item[(b)] For every nonempty subset $F$ of $[k]^{<m+1}$ we have $\mathrm{I}_W(\wedge F)=\wedge \mathrm{I}_W(F)$. In particular, 
Carlson--Simpson trees preserve infima.
\end{enumerate}
\end{fact}
By Fact \ref{f25}, for most practical purposes we may identify an $m$-dimensional Carlson--Simpson tree $W$ with $[k]^{<m+1}$
via the canonical isomorphism $\mathrm{I}_W$.

We close this subsection with the following consequence of Theorem \ref{t23}.
\begin{lem} \label{l26}
Let $k,d,r$ be positive integers with $k\meg 2$. If $W$ is a Carlson--Simpson tree $W$ of $[k]^{<\nn}$ with
\begin{equation} \label{e213}
\dim(W)\meg \mathrm{CS}(k,d+1,1,r),
\end{equation}
then for every $r$-coloring of $W$ there exists $V\in\mathrm{Subtr}_d(W)$ which is monochromatic.
\end{lem}
\begin{proof}
Let $\mathbf{w}$ be the generating sequence of $W$ and fix an $r$-coloring $c: W\to [r]$. Define 
$\tilde{c}:\mathrm{Subseq}_1(\mathbf{w})\to [r]$ by $\widetilde{c}\big( (v,v_0)\big)= c(v)$ for every
$(v,v_0)\in \mathrm{Subseq}_1(\mathbf{w})$. Invoking Theorem \ref{t23} and \eqref{e213}, we may select 
$\mathbf{v}=(v,v_0,...,v_d)\in\mathrm{Subseq}_{d+1}(\mathbf{w})$ such that the set $\mathrm{Subseq}_1(\mathbf{v})$ is monochromatic
with respect to the coloring $\tilde{c}$. Let $V$ be the $d$-dimensional Carlson--Simpson tree generated by the sequence
$(v,v_0,...,v_{d-1})$. It is easy to see that $V$ is as desired.
\end{proof}

\subsection{Probabilistic preliminaries}

We recall the definition of a class of probability measures on $[k]^{<\nn}$ introduced by H. Furstenberg and B. Weiss in \cite{FW}.
Specifically, let $k,m\in\nn$ with $k\meg 2$ and $m\meg 1$. The \textit{Furstenberg--Weiss measure} $\mathrm{d}_{\mathrm{FW}}^m$ associated
to $[k]^{<m+1}$ is the probability measure on $[k]^{<\nn}$ defined by
\begin{equation} \label{e214}
\mathrm{d}_{\mathrm{FW}}^m(A)= \ave_{n\in\{0,...,m\}}\frac{A\cap[k]^n}{k^n}.
\end{equation}
We will also need the following standard lemma.
\begin{lem} \label{l27}
Let $0<\vartheta<\ee\mik 1$ and $n\in\nn$ with $n\meg (\ee^2-\vartheta^2)^{-1}$. If $(A_i)_{i=1}^{n}$ is a family of measurable events
in a probability space $(\Omega,\Sigma,\mu)$ satisfying $\mu(A_i)\meg \ee$ for every $i\in [n]$, then there exist $i,j\in [n]$ with $i\neq j$
such that $\mu(A_i\cap A_j) \meg \vartheta^2$.
\end{lem}
\begin{proof}
We set $X=\sum_{i=1}^n \mathbf{1}_{A_i}$ where $\mathbf{1}_{A_i}$ is the indicator function of the event $A_i$ for every $i\in [n]$.
Then $\mathbb{E}[X]\meg \ee n$ so, by convexity,
\begin{equation} \label{e215}
\sum_{i\in [n]} \sum_{j\in [n]\setminus \{i\}} \mu(A_i\cap A_j) = \mathbb{E}[X(X-1)] \meg \ee n (\ee n-1).
\end{equation}
Therefore, there exist $i,j\in [n]$ with $i\neq j$ such that $\mu(A_i\cap A_j)\meg \theta^2$.
\end{proof}

\subsection{The density Carlson--Simpson Theorem}

We will need the following result (see \cite[Theorem B]{DKT3}).
\begin{thm} \label{t28}
For every integer $k\meg 2$, every integer $m\meg 1$ and every $0<\delta\mik 1$ there exists an integer $N$ with the following property. If $L$
is a finite subset of $\nn$ of cardinality at least $N$ and $A$ is a set of words over $k$ satisfying $|A\cap [k]^n|\meg \delta k^n$ for every
$n\in L$, then there exists a Carlson--Simpson sequence $(w,w_0,...w_{m-1})$ over $k$ such that the set
\begin{equation} \label{e216}
\{w\}\cup \big\{w^{\con}w_0(a_0)^{\con}...^{\con}w_n(a_n): n\in\{0,...,m-1\} \text{ and } a_0,...,a_n\in [k]\big\}
\end{equation}
is contained in $A$. The least integer $N$ with this property will be denoted by $\dcs(k,m,\delta)$.
\end{thm}
Theorem \ref{t28} is the density version of a well-known coloring result due to T. J. Carlson and S. J. Simpson \cite{CS}.
Also we notice that the argument in \cite{DKT3} is effective and gives explicit upper bounds for the numbers $\dcs(k,m,\delta)$. These
upper bounds, however, have an Ackermann-type dependence with respect to $k$.

The final result of this subsection is a consequence of Theorem \ref{t28}. To state it we need, first, to introduce some numerical invariants.
Specifically, for every integer $k\meg 2$ and every $0< \delta\mik 1$ we set
\begin{equation} \label{e217}
\Lambda=\Lambda(k,\delta)=\lceil4\delta^{-1}\mathrm{DCS}(k,1,\delta/4)\rceil
\end{equation}
and
\begin{equation} \label{e218}
\eta(k,\delta)=\frac{\delta}{2|\mathrm{Subtr}_1([k]^{<\Lambda})|}.
\end{equation}
We have the following lemma.
\begin{lem} \label{l29}
Let $k\in \nn$ with $k\meg 2$ and $0<\delta\mik 1$, and define $\Lambda=\Lambda(k,\delta)$ as in \eqref{e217}. Also let 
$\{A_t: t\in [k]^{<\Lambda}\}$ be a family of measurable events in a probability space $(\Omega,\Sigma, \mu)$ satisfying $\mu(A_t)\meg \delta$
for every $t\in [k]^{<\Lambda}$. Then there exists a Carlson--Simpson line $S$ of $[k]^{<\Lambda}$ such that
\begin{equation} \label{e219}
\mu\Big(\bigcap_{t\in S}A_t\Big)\meg \eta(k,\delta)
\end{equation}
where $\eta(k,\delta)$ is as in \eqref{e218}.
\end{lem}
We remark that Lemma \ref{l29} follows from \cite[Lemma 7.9]{DKT3}. It is based on an argument that can be traced in an old paper of P. Erd\H{o}s
and A. Hajnal \cite{EH}. For the convenience of the reader we include the proof.
\begin{proof}[Proof of Lemma \ref{l29}]
Let $\mathrm{d}_{\mathrm{FW}}^{\Lambda-1}$ be the Furstenberg--Weiss measure associated to $[k]^{<\Lambda}$ and set
\begin{equation} \label{e220}
Y=\Big\{ \omega\in\Omega: \mathrm{d}_{\mathrm{FW}}^{\Lambda-1}\big(\{t\in [k]^{<\Lambda}: \omega\in A_t\}\big)\meg \delta/2\Big\}.
\end{equation}
Since $\mu(A_t)\meg \delta$ for every $t\in [k]^{<\Lambda}$ we have $\mu(Y)\meg \delta/2$.

Let $\omega\in Y$ be arbitrary and set $A_\omega=\{t\in [k]^{<\Lambda}: \omega\in A_t\}$. Observe that 
\begin{equation} \label{e221}
|\big\{n\in\{0,...,\Lambda-1\}: \frac{|A_{\omega}\cap [k]^n|}{k^n}\meg \delta/4\big\}| \meg (\delta/4)\Lambda
\stackrel{\eqref{e217}}{\meg} \mathrm{DCS}(k,1,\delta/4).
\end{equation}
By Theorem \ref{t28}, there exists a Carlson--Simpson line $S_\omega$ of $[k]^{<\nn}$ with $S_{\omega}\subseteq A_\omega$.
In particular, we have $S_\omega\subseteq [k]^{<\Lambda}$ and 
\begin{equation} \label{e222}
\omega\in \bigcap_{t\in S_\omega}A_t.
\end{equation}
By the classical pigeonhole principle, there exist $Z\in\Sigma$ and a Carlson--Simpson line $S$ of $[k]^{<\Lambda}$ such that
$S_\omega=S$ for every $\omega\in Z$ and 
\begin{equation} \label{e223}
\mu(Z)\meg\frac{\mu(Y)}{|\mathrm{Subtr}_1([k]^{<\Lambda})|}
\meg\frac{\delta/2}{|\mathrm{Subtr}_1([k]^{<\Lambda})|}\stackrel{\eqref{e218}}{=}\eta(k,\delta).
\end{equation}
Hence,
\begin{equation} \label{e224}
\mu\Big( \bigcap_{t\in S}A_t \Big) \meg \mu(Z)\stackrel{\eqref{e223}}{\meg} \eta(k,\delta)
\end{equation}
and the proof is completed.
\end{proof}


\section{Proof of Theorem \ref{t14}}

\numberwithin{equation}{section}

This section is devoted to the proof of Theorem \ref{t14} stated in the introduction. It is organized as follows. In \S 3.1 we introduce some
pieces of notation and isolate some basic properties of types, while in \S 3.2 we gather some preliminary tools. The proof of Theorem \ref{t14}
is completed in \S 3.3.

\subsection{Types: definitions and basic properties} 

Let $k,q\in\nn$ with $k\meg 2$ and $q\meg 1$. Also let $\mathcal{L}$ be an alphabet of cardinality $q$ which is disjoint from $[k]$ and denote
by $\mathrm{W}(k,\mathcal{L})$ the set of all words over $[k]\cup \mathcal{L}$. Recall that a \textit{type} of $\mathrm{W}(k,\mathcal{L})$ is a
nonempty word over $\mathcal{L}$ having no consecutive multiple appearances of the same letter. The set of all types of $\mathrm{W}(k,\mathcal{L})$
will be denoted by $\mathcal{T}(\mathcal{L})$. Also recall that for every $w\in \mathrm{W}(k,\mathcal{L})\setminus [k]^{<\nn}$ its type
is defined as follows. First we erase all letters of $w$ which belong to $[k]$, then we shorten the runs of the same letters of $\mathcal{L}$
to singletons and, finally, we push everything back together. For every $\tau\in\mathcal{T}(\mathcal{L})$ by $\mathrm{W}(k,\mathcal{L},\tau)$
we shall denote the set of all $w\in \mathrm{W}(k,\mathcal{L})$ of type $\tau$. Hence,
\begin{equation} \label{e31}
\mathrm{W}(k,\mathcal{L})\setminus[k]^{<\nn}= \bigcup_{\tau\in\mathcal{T}(\mathcal{L})}\mathrm{W}(k,\mathcal{L},\tau).
\end{equation}
Let $\mathbf{w}=(w, w_0,...,w_{n-1})$ be an $n$-dimensional Carlson--Simpson sequence over $k$. For every $v=(a_0,...,a_{m-1})
\in \mathrm{W}(k,\mathcal{L})$ of length at most $n$ we set
\begin{equation} \label{e32}
\mathbf{w}(v)=w^\con w_0(a_0)^\con w_1(a_1)^\con...^\con w_{m-1}(a_{m-1})
\end{equation}
with the convention that $\mathbf{w}(v)=w$ if $v$ is the empty word. We define
\begin{equation} \label{e33}
\mathbf{w}(k,\mathcal{L})=\big\{\mathbf{w}(v):  v \in \mathrm{W}(k,\mathcal{L}) \text{ of length at most } \dim(\mathbf{w})\big\}.
\end{equation}
Moreover, for every $\tau\in \mathcal{T}(\mathcal{L})$ we set
\begin{equation} \label{e34}
\mathbf{w}(k,\mathcal{L},\tau)=\mathbf{w}(k,\mathcal{L})\cap \ \mathrm{W}(k,\mathcal{L},\tau).
\end{equation}
We will need the following elementary fact.
\begin{fact} \label{f31}
Let $\mathbf{w}$ be an $n$-dimensional Carlson--Simpson sequence over $k$.
\begin{enumerate}
\item[(a)] For every $\tau\in\mathcal{T}(\mathcal{L})$ the set  $\mathbf{w}(k,\mathcal{L},\tau)$ is nonempty if and only if $|\tau|\mik n$.
Moreover, if $W$ is the Carlson--Simpson tree generated by $\mathbf{w}$, then
\begin{equation} \label{e35}
\mathbf{w}(k,\mathcal{L})\setminus W=\bigcup_{\{\tau\in\mathcal{T}(\mathcal{L}):|\tau|\mik n\}}\mathbf{w}(k,\mathcal{L},\tau).
\end{equation}
\item[(b)] If $\mathbf{v}$ is a Carlson--Simpson subsequence of $\mathbf{w}$, then
\begin{equation} \label{e36}
\mathbf{v}(k,\mathcal{L})\subseteq \mathbf{w}(k,\mathcal{L}).
\end{equation} 
Moreover, for every $\tau\in\mathcal{T}(\mathcal{L})$ we have
\begin{equation} \label{e37}
\mathbf{v}(k,\mathcal{L},\tau)\subseteq \mathbf{w}(k,\mathcal{L},\tau).
\end{equation}
\end{enumerate}
\end{fact}

\subsection{Preliminary tools}

We have the following lemma.
\begin{lem} \label{l32} 
Let $k,n\in\nn$ with $k\meg 2$ and $n\meg 1$. Also let $\mathcal{L}$ be a finite alphabet which is disjoint from $[k]$ and $\mathbf{w}$ be a
Carlson--Simpson sequence over $k$ of dimension $n$. Then for every $m\in [n]$ and every $\tau\in\mathcal{T}(\mathcal{L})$ with $|\tau|=m$ we have
\begin{equation} \label{e38}
\mathbf{w}(k,\mathcal{L},\tau)=\big\{\mathbf{z}(\tau): \mathbf{z}\in \mathrm{Subseq}_m(\mathbf{w})\big\}.
\end{equation}
\end{lem}
\begin{proof} 
Write $\mathbf{w}=(w,w_0,...,w_{n-1})$ and fix $m\in [n]$ and a type $\tau=(\lambda_0,...,\lambda_{m-1})$. By \eqref{e32}, we see
that $\mathbf{z}(\tau)$ is of type $\tau$ for every $\mathbf{z}\in \mathrm{Subseq}_m(\mathbf{w})$. Conversely, let $v=(a_0,...,a_{\ell-1})\in
\mathrm{W}(k,\mathcal{L})$ such that $\mathbf{w}(v)\in \mathbf{w}(k,\mathcal{L},\tau)$. We need to find an $m$-dimensional Carlson--Simpson
subsequence $\mathbf{z}$ of $\mathbf{w}$ such that $\mathbf{z}(\tau)=\mathbf{w}(v)$. Observe that the words $v$ and $\mathbf{w}(v)$
are of the same type, and therefore, $v$ is of type $\tau$. We define $(n_i)_{i=0}^{m-1}$ by the rule 
$n_0=\min\big\{j\in\{0,...,\ell-1\}:a_j=\lambda_0\big\}$ and
\begin{equation} \label{e39}
n_{i+1}=\min\big\{j\in\{n_i+1,...,\ell-1\}: a_j=\lambda_{i+1}\big\}. 
\end{equation}
Notice that $0\mik n_0< ...< n_{m-1}< \ell\mik n$. For every $i\in\{0,...,m-1\}$ we set $a'_i=a_i$ if $a_i\in [k]$; otherwise we set
$a'_i=x$. We define an $m$-dimensional Carlson--Simpson subsequence $\mathbf{z}=(z,z_0,...,z_{m-1})$ of $\mathbf{w}$ as follows. First we set
\begin{equation} \label{e310}
z=w^\con w_0(a'_0)^\con ...^\con w_{n_0-1}(a'_{n_0-1})
\end{equation} 
with the convention that $z=w$ if $n_0=0$. Next for every $i\in\{0,..., m-1\}$ let
\begin{equation} \label{e311}
z_i=w_{n_i}(a'_{n_i})^\con ...^\con w_{n_{i+1}-1}(a'_{n_{i+1}-1}).
\end{equation}  
It is easily checked that $\mathbf{z}$ is an $m$-dimensional Carlson--Simpson subsequence of $\mathbf{w}$ and satisfies
$\mathbf{z}(\tau)=\mathbf{w}(v)$. This shows that $\mathbf{w}(k,\mathcal{L},\tau)\subseteq \{\mathbf{z}(\tau): \mathbf{z}
\in \mathrm{Subseq}_m(\mathbf{w})\}$ and the proof is completed.
\end{proof}
Lemma \ref{l32} has the following consequence. It will enable us to reduce the proof of Theorem \ref{t14} to Theorem \ref{t23}. 
\begin{cor} \label{c33}
Let $k$ and $\mathcal{L}$ be as in Lemma \ref{l32}. Also let $d$ be a positive integer and $\mathbf{w}=(w,w_0,...,w_{2d-1})$ be a
Carlson--Simpson sequence over $k$ of dimension $2d$. Then, setting $\mathbf{w}'=(w,w_0,...,w_{d-1})$, for every $\tau\in\mathcal{T}(\mathcal{L})$ 
of length at most $d$ we have 
\begin{equation} \label{e312}
\mathbf{w}'(k,\mathcal{L},\tau)\subseteq \big\{ \mathbf{z}(\tau): \mathbf{z}\in\mathrm{Subseq}_d(\mathbf{w})\big\}.
\end{equation}
\end{cor}
\begin{proof}
We start with two elementary observations. First notice that if $\mathbf{z}, \mathbf{z'}$ are Carlson--Simpson sequences over $k$
with $\mathbf{z}'$ an initial segment of $\mathbf{z}$, then for every $v\in W(k,\mathcal{L})$ of length at most $\dim(\mathbf{z}')$
we have  $\mathbf{z}(v)=\mathbf{z}'(v)$. Next observe that for every Carlson--Simpson subsequence $\mathbf{z}'$ of $\mathbf{w}'$ there exists
a (not necessarily unique) Carlson--Simpson subsequence $\mathbf{z}$ of $\mathbf{w}$ with $\dim(\mathbf{z})=d$ and such that $\mathbf{z}'$
is an initial segment of $\mathbf{z}$.

Now fix a type $\tau$ of length at most $d$ and set $m=|\tau|$. By Lemma \ref{l32} and using the previous remarks, we see that
\begin{equation} \label{e313}
\mathbf{w}'(k,\mathcal{L},\tau) = \big\{\mathbf{z}'(\tau): \mathbf{z}'\in \mathrm{Subseq}_m(\mathbf{w}')\big\} 
\subseteq \big\{\mathbf{z}(\tau): \mathbf{z}\in \mathrm{Subseq}_d(\mathbf{w})\big\}
\end{equation}
as desired.
\end{proof}

\subsection{Proof of Theorem \ref{t14}}

We claim that 
\begin{equation} \label{e314}
\mathrm{FK}(k,q,d,r)\mik \mathrm{CS}\big(k,2d,d, r^{q^{d+1}}\big)
\end{equation}
for every $k,q,d,r\in\nn$ with $k\meg 2$ and $q,d,r\meg 1$. Indeed, let $\mathcal{L}$ be an alphabet of cardinality $q$ which is disjoint from 
$[k]$. Also let $\mathbf{w}$ be a Carlson--Simpson sequence over $k$ with $\dim(\mathbf{w})\meg \mathrm{CS}\big(k,2d,d, r^{q^{d+1}}\big)$
and fix an $r$-coloring $c:\mathbf{w}(k,\mathcal{L})\to [r]$. We set $T=\{\tau\in \mathcal{T}(\mathcal{L}): |\tau|\mik d\}$. We define a
coloring $\tilde{c}:\mathrm{Subseq}_d(\mathbf{w})\to [r]^T$ by 
\begin{equation} \label{e315}
\tilde{c}(\mathbf{z})= \left\langle c( \mathbf{z}(\tau)\big): \tau \in T \right\rangle.
\end{equation}
By Fact \ref{f31}, the coloring $\tilde{c}$ is well-defined. Also notice that
\begin{equation} \label{e316}
|T|\mik \sum_{i=1}^d q^i \mik q^{d+1}.
\end{equation}
Hence, by Theorem \ref{t23}, there exists a Carlson--Simpson subsequence $\mathbf{u}$ of $\mathbf{w}$ of dimension $2d$ such that the set 
$\mathrm{Subseq}_d(\mathbf{u})$ is monochromatic with respect to $\tilde{c}$. Let $\mathbf{v}$ be the unique Carlson--Simpson sequence
over $k$ of dimension $d$ which is an initial segment of $\mathbf{u}$. Notice, in particular, that $\mathbf{v}$ is a Carlson--Simpson
subsequence of $\mathbf{w}$. Fix a type $\tau$. If $\tau\notin T$, then by Fact \ref{f31} the set $\mathbf{v}(k,\mathcal{L},\tau)$ is empty.
Otherwise, by Corollary \ref{c33} and the choice of $\mathbf{u}$, the set $\mathbf{v}(k,\mathcal{L},\tau)$ is monochromatic. This shows that
\eqref{e314} is satisfied.

Finally, by Corollary \ref{c24} and taking into account this estimate, we see that the numbers $\mathrm{FK}(k,q,d,r)$ are bounded
by a primitive recursive function belonging to the class $\mathcal{E}^6$. The proof of Theorem \ref{t14} is thus completed.


\section{Flat sets}

\numberwithin{equation}{section}

\subsection{Motivation}

As we have mentioned in the introduction, there is no analogue of Ramsey's classical Theorem for colorings of pairs of $[k]^{<\nn}$ (or, more
generally, of subsets of $[k]^{<\nn}$ of a given cardinality). Nevertheless, there is non-trivial information on this Ramsey-type problem.
The relevant tools will be developed in this and the next two sections. Specifically, we shall define certain classes of finite subsets of
$[k]^{<\nn}$ with the following crucial properties. Firstly, each class is partition regular. Secondly, their union is sufficiently ``dense"
in the sense that for every nonempty finite subset $F$ of $[k]^{<\nn}$ one can find an element of one of the classes that contains $F$ and
whose cardinality is effectively controlled by the cardinality of $F$.

In this section we make the first step towards this goal. In particular, we define the family of \textit{flat} sets which are the building
blocks of the elements of the classes mentioned above. Their properties are discussed in \S 4.3 and \S 4.4.

Finally we note that the aforementioned analysis and the corresponding tools developed in this paper, form the basis for a complete classification 
of those classes of finite subsets of Carlson--Simpson trees which are partition regular. The details of this classification will appear elsewhere.

\subsection{Definition}

Let $k\in\nn$ with $k\meg 2$. A finite subset $F$ of $[k]^{<\nn}$ will be called \textit{flat} provided that (a) $|F|\meg 2$, and (b) there
exists $n\in\nn$ such that $F\subseteq [k]^n$. By $\mathrm{Fl}([k]^{<\nn})$ we shall denote the set of all flat subsets of $[k]^{<\nn}$. If 
$p\in\nn$ with $p\meg 2$, then $\mathrm{Fl}_p([k]^{<\nn})$ stands for the family of all flat sets of cardinality $p$. Moreover, for every 
Carlson--Simpson tree $W$ of $[k]^{<\nn}$ by $\mathrm{Fl}(W)$ we shall denote the set of all flat subsets of $[k]^{<\nn}$ which are contained
in $W$. The set $\mathrm{Fl}_p(W)$ is analogously defined. Notice, in particular, that $\mathrm{Fl}_p(W)\neq\varnothing$ if and only if
$2\mik p\mik k^{\dim(W)}$. Therefore,
\begin{equation} \label{e41}
\mathrm{Fl}(W)=\bigcup_{p=2}^{k^{\dim(W)}}\mathrm{Fl}_p(W).
\end{equation}

\subsection{Word representation of flat sets and their type}

Let $k$ and $p$ be a pair of integers with $k,p\meg 2$. These parameters will be fixed throughout this subsection. We set
\begin{equation} \label{e42}
\Delta([k]^p)=\big\{(\underbrace{a,...,a}_{p-\mathrm{times}}): a\in [k]\big\}
\end{equation}
and 
\begin{equation} \label{e43}
\mathcal{L}_{p}=[k]^p\setminus  \Delta([k]^p).
\end{equation}  
As in \S 3.1, let $\mathrm{W}(k,\mathcal{L}_p)$ be the set of all words over $[k]\cup \mathcal{L}_p$. We shall define a map
\begin{equation} \label{e44}
\mathrm{R}_{p}:\mathrm{Fl}_{p}([k]^{<\nn})\to \mathrm{W}(k, \mathcal{L}_p)
\end{equation}
as follows. Let $F\in\mathrm{Fl}_p([k]^{<\nn})$ be arbitrary and write the set $F$ in lexicographical increasing order as
$\{t_0<_{\text{lex}}...<_{\text{lex}} t_{p-1}\}$. Also let $n\in\nn$ be such that $F\subseteq [k]^{n}$ and notice that $n\meg 1$.
For every $i\in\{0,...,n-1\}$ and every $j\in\{0,...,p-1\}$ let $a_{i,j}$ be the $i$-th coordinate of $t_j$ and set
$\overline{a}_i=(a_{i,0},...,a_{i,p-1})$. Observe that $\overline{a}_i\in[k]^p$ for every $i\in\{0,...,n-1\}$. We set
\begin{equation} \label{e45}
\tilde{a}_i =
\begin{cases}
a & \text{if \ } \overline{a}_i=(\underbrace{a,...,a}_{p-\mathrm{times}})\in \Delta([k]^p), \\
\overline{a}_i & \text{if \ } \overline{a}_i\in [k]^p\setminus\Delta([k]^p).
\end{cases}
\end{equation}
Finally we define 
\begin{equation} \label{e46}
\mathrm{R}_{p}(F)= (\tilde{a}_0,...,\tilde{a}_{n-1}).
\end{equation}
We call the word $\mathrm{R}_p(F)$ the \textit{word representation} of the flat set $F$.

Before we proceed let us give a specific example. Let $k=2$ and $p=4$ and consider the subset $F$ of $[2]^3$ consisting of the elements
$(1,2,2)$, $(1,1,2)$, $(2,1,2)$ and $(2,2,2)$. We order $F$ lexicographically as
\begin{equation} \label{e47}
F=\big\{ (1,1,2)<_{\mathrm{lex}} (1,2,2)<_{\mathrm{lex}} (2,1,2)<_{\mathrm{lex}} (2,2,2) \big\}
\end{equation} 
and we observe that
\begin{equation} \label{e48}
\overline{a}_0=(1,1,2,2), \ \overline{a}_1=(1,2,1,2) \ \text{ and } \ \overline{a}_2=(2,2,2,2).
\end{equation} 
Thus, $\mathrm{R}_4(F)$ is the word $\big((1, 1,  2,  2), (1, 2, 1, 2), 2\big)\in \mathrm{W}(2,\mathcal{L}_4)$.

Now let $F\in\mathrm{Fl}_p([k]^{<\nn})$. We define the \textit{type} of $F$ to be the type of its word representation $\mathrm{R}_p(F)$ in
$\mathrm{W}(k,\mathcal{L}_p)$. In particular, the type of $F$ is an element of $\mathcal{T}(\mathcal{L}_p)$. Notice that for every integer
$p'\meg 2$ with $p'\neq p$ we have that $\mathcal{L}_{p'}\cap \mathcal{L}_p=\varnothing$, and so, $\mathcal{T}(\mathcal{L}_{p'})\cap
\mathcal{T}(\mathcal{L}_p)=\varnothing$. Therefore, if $F$ and $F'$ are flat sets of different cardinality, then their types are different.

For every $\tau\in\mathcal{T}(\mathcal{L}_p)$ by $\mathrm{Fl}_{p,\tau}([k]^{<\nn})$ we shall denote the set of all $F\in\mathrm{Fl}_{p}([k]^{<\nn})$ 
of type $\tau$. It is easy to see that if $\mathrm{Fl}_{p,\tau}([k]^{<\nn})$ is nonempty, then $p\mik k^{|\tau|}$. Notice, however, that not 
every $\tau\in\mathcal{T}(\mathcal{L}_p)$ with $p\mik k^{|\tau|}$ is realized as the type of a flat set of cardinality $p$. This is due to the
fact that the type of a flat set $F$ is determined after we have ordered $F$ lexicographically. Taking into account these remarks, we set
\begin{equation} \label{e49}
\mathcal{T}_{\mathrm{Fl}}[k,p]=\big\{\tau\in\mathcal{T}(\mathcal{L}_p): \mathrm{Fl}_{p,\tau}([k]^{<\nn})\neq\varnothing\big\}.
\end{equation}
Of course, we can relativize the above definitions to Carlson--Simpson trees. Specifically, for every Carlson--Simpson tree $W$ of $[k]^{<\nn}$
and every $\tau\in\mathcal{T}(\mathcal{L}_p)$ by $\mathrm{Fl}_{p,\tau}(W)$ we shall denote the set of all $F\in\mathrm{Fl}_{p,\tau}([k]^{<\nn})$
which are contained in $W$. Observe that $\mathrm{Fl}_{p,\tau}(W)$ is nonempty if and only if $\tau\in\mathcal{T}_{\mathrm{Fl}}[k,p]$ and
$|\tau|\mik \dim(W)$. 

We summarize, below, some basic properties of the map $\mathrm{R}_p$. 
\begin{fact} \label{f41}
Let $k,p\in\nn$ with $k,p\meg2$ and $\tau\in\mathcal{T}_{\mathrm{Fl}}[k,p]$. Then the following hold.
\begin{enumerate}
\item[(a)] The map $\mathrm{R}_{p}:\mathrm{Fl}_p([k]^{<\nn})\to \mathrm{W}(k, \mathcal{L}_p)$ is an injection.
\item[(b)] The restriction of $\mathrm{R}_{p}$ to $\mathrm{Fl}_{p,\tau}([k]^{<\nn})$ is onto $\mathrm{W}(k,\mathcal{L}_p,\tau)$.
In particular, the map $\mathrm{R}_{p}:\mathrm{Fl}_{p,\tau}([k]^{<\nn})\to \mathrm{W}(k,\mathcal{L}_p,\tau)$ is a bijection.
\end{enumerate}
\end{fact}
Fact \ref{f41} is a straightforward consequence of the relevant definitions. We will also need the following elementary fact. 
\begin{fact} \label{f42}
Let $k,p\in\nn$ with $k,p\meg 2$. Also let $V$ be a Carlson--Simpson tree of $[k]^{<\nn}$ and denote by $\mathbf{v}$ the Carlson--Simpson
sequence that generates $V$. Finally let $\tau\in\mathcal{T}(\mathcal{L}_p)$ such that the set $\mathrm{Fl}_{p,\tau}(V)$ is nonempty. Then
the following are satisfied.
\begin{enumerate}
\item[(a)] Setting $m=\dim(V)$, we have that $\mathrm{Fl}_{p,\tau}(V)=\mathrm{I}_V\big(\mathrm{Fl}_{p,\tau}([k]^{<m+1})\big)$ where $\mathrm{I}_V$
is the canonical isomorphism associated to $V$.
\item[(b)] The restriction of $\mathrm{R}_{p}$ to $\mathrm{Fl}_p(V)$ is an injection into $\mathbf{v}(k,\mathcal{L}_p)$.
\item[(c)] The restriction of $\mathrm{R}_{p}$ to $\mathrm{F}_{p,\tau}(V)$ is onto $\mathbf{v}(k,\mathcal{L}_p,\tau)$.
In particular, the map $\mathrm{R}_{p}:\mathrm{Fl}_{p,\tau}(V)\to \mathbf{v}(k,\mathcal{L}_p,\tau)$ is a bijection.
\end{enumerate}
\end{fact}

\subsection{Ramsey properties of flat sets}

We have the following proposition.

\begin{prop} \label{p43}
Let $k,d,r\in\nn$ with $k\meg 2$ and $d,r\meg 1$. Then there exists a positive integer $N$ with the following property. If $n\meg N$, then
for every $n$-dimensional Carlson--Simpson tree $W$ of $[k]^{<\nn}$ and every $r$-coloring of $\mathrm{Fl}(W)$ there exists a $d$-dimensional
Carlson--Simpson subtree $V$ of $W$ such that every pair of flat subsets of $V$ with the same type is monochromatic.
\end{prop}
\begin{proof} 
We set
\begin{equation} \label{e410}
q=\sum_{p=2}^{k^d}(k^p-k).
\end{equation}
We will show that the desired positive integer $N$ can be chosen to be the number $\mathrm{FK}(k,q,d,r)$. To this end let 
$n\meg\mathrm{FK}(k,q,d,r)$ be arbitrary. Also let $W$ be an $n$-dimensional Carlson--Simpson tree of $[k]^{<\nn}$ and denote
by $\mathbf{w}$ its generating sequence. Fix a coloring $c:\mathrm{Fl}(W)\to [r]$. Let
\begin{equation} \label{e411}
\mathcal{L}=\bigcup_{p=2}^{k^d}\mathcal{L}_p.
\end{equation}
By the definition of the finite alphabet $\mathcal{L}_p$ in \eqref{e43} and the choice of $q$ in \eqref{e410}, we see that
$|\mathcal{L}|=q$. Next let $\mathcal{F}$ be the subset of $\mathrm{Fl}(W)$ defined by
\begin{equation} \label{e412}
F\in\mathcal{F} \Leftrightarrow \text{the length of the type of $F$ is at most $d$.}
\end{equation}
Notice that $2\mik |F|\mik k^d$ for every $F\in\mathcal{F}$. We define $\mathrm{R}: \mathcal{F}\to \mathbf{w}(k,\mathcal{L})$ by
the rule $\mathrm{R}(F)=\mathrm{R}_{|F|}(F)$. It is easy to see that the map $\mathrm{R}$ is a well-defined injection.

Finally let $\tilde{c}:\mathbf{w}(k,\mathcal{L})\to [r]$ be defined by
\begin{equation} \label{e413}
\tilde{c}(w)=
\begin{cases}
c(F)  & \text{if there exists } F\in\mathcal{F} \text{ with } \mathrm{R}(F)=w, \\
r     & \text{otherwise}. \\
\end{cases}
\end{equation}
Since $\dim(\mathbf{w})\meg \mathrm{FK}(k,q,d,r)$, by Theorem \ref{t14}, there exists $\mathbf{v}\in\mathrm{Subseq}_d(\mathbf{w})$
such that the set $\mathbf{v}(k,\mathcal{L},\tau)$ is monochromatic with respect to $\tilde{c}$ for every $\tau\in\mathcal{T}(\mathcal{L})$
with $|\tau|\mik d$. Let $V$ be the Carlson--Simpson tree generated by $\mathbf{v}$. By Fact \ref{f42}, we see that $V$ is as desired.
\end{proof}


\section{Basic sets}

\numberwithin{equation}{section}

In this section we continue the analysis outlined in \S 4.1. We start with the following definition.
\begin{defn} \label{d51}
A \emph{basic set} of $[k]^{<\nn}$ is a subset of $[k]^{<\nn}$ of the form
\begin{equation} \label{e51}
B= \{s\}\cup \bigcup_{i=0}^{\ell-1} (s^\con F_0^\con ... ^\con F_i)
\end{equation}
where $\ell$ is a positive integer, $s$ is a word over $k$ and for every $i\in \{0,...,\ell-1\}$ the set $F_i$ is a flat subset of $[k]^{<\nn}$
whose infimum $\wedge F_i$ is the empty word.
\end{defn}
Notice that the sequence $(s, F_0,..., F_{\ell-1})$ that generates a basic set $B$ via formula \eqref{e51} is unique. It will be called the
\textit{generating sequence} of $B$. The word $s$ will be called the \textit{top} of $B$, while the positive integer $\ell$ will be called the
\textit{dimension} of $B$ and will be denoted by $\dim(B)$. 

Let $B$ be a basic set of $[k]^{<\nn}$ and $(s,F_0,...,F_{\ell-1})$ be its generating sequence. The \textit{level width sequence} of $B$, denoted
by $\textbf{p}(B)$, is defined to be the sequence $(|F_0|,...,|F_{\ell-1}|)$. The integer $\max_{0\mik i\mik\ell-1}|F_i|$ will be called the
\textit{width} of $B$ and will be denoted by $w(B)$. Finally, the \textit{type} of $B$ is defined to be the sequence $(\tau_0,...,\tau_{\ell-1})$
where $\tau_i$ is the type of the flat set $F_i$ for every $i\in\{0,...,\ell-1\}$. It will be denoted by $\mathbf{t}(B)$. Notice the following
rigidity property of the type of $B$: if $B'$ is another basic set with $\mathbf{t}(B')=\mathbf{t}(B)$, then $\dim(B')=\dim(B)$ and
$\mathbf{p}(B')=\mathbf{p}(B)$.

By $\mathrm{B}([k]^{<\nn})$ we shall denote the set of all basic sets of $[k]^{<\nn}$. The set of all basic sets of $[k]^{<\nn}$ which are
contained in a Carlson--Simpson tree $W$ of $[k]^{<\nn}$ will be denoted by $\mathrm{B}(W)$.

Now let $\ell$ be a positive integer and $\mathbf{p}=(p_0,...,p_{\ell-1})$ be a finite sequence in $\nn$ with $p_i\meg 2$ for every
$i\in\{0,...,\ell-1\}$. Also let $\mathbf{t}=(\tau_0,...,\tau_{\ell-1})$ be such that $\tau_i\in\mathcal{T}(\mathcal{L}_{p_i})$ for every
$i\in\{0,...,\ell-1\}$, where $\mathcal{L}_{p_i}$ is as in \eqref{e43}. We set
\begin{equation} \label{e52}
\mathrm{B}_{\ell,\mathbf{p},\mathbf{t}}([k]^{<\nn})= \big\{ B\in \mathrm{B}([k]^{<\nn}): \dim(B)=\ell, \mathbf{p}(B)=\mathbf{p} 
\text{ and } \mathbf{t}(B)=\mathbf{t} \big\}.
\end{equation}
For every Carlson--Simpson tree $W$ of $[k]^{<\nn}$ the set $\mathrm{B}_{\ell,\mathbf{p},\mathbf{t}}(W)$ is analogously defined. Observe that
the set $\mathrm{B}_{\ell,\mathbf{p},\mathbf{t}}([k]^{<\nn})$ is nonempty if and only if $\tau_i\in\mathcal{T}_{\mathrm{Fl}}[k,p_i]$ for every
$i\in\{0,...,\ell-1\}$, where $\mathcal{T}_{\mathrm{Fl}}[k,p_i]$ is as in \eqref{e49}. Respectively, if $W$ is a Carlson--Simpson tree $W$ of
$[k]^{<\nn}$, then the set $\mathrm{B}_{\ell,\mathbf{p},\mathbf{t}}(W)$ is nonempty if and only if $\tau_i\in\mathcal{T}_{\mathrm{Fl}}[k,p_i]$
for every $i\in\{0,...,\ell-1\}$ and $\sum_{i=0}^{\ell-1}|\tau_i|\mik \dim(W)$.

Much of our interest on basic sets stems from the fact that they possess strong structural properties. In particular, we have the following
theorem which is the main result of this section.
\begin{thm} \label{t52}
Let $k,d,r\in\nn$ with $k\meg 2$ and $d,r\meg 1$. Then there exists a positive integer $N$ with the following property. If $n\meg N$, then
for every $n$-dimensional Carlson--Simpson tree $W$ of $[k]^{<\nn}$ and every $r$-coloring of $\mathrm{B}(W)$ there exists a $d$-dimensional
Carlson--Simpson subtree $U$ of $W$ such that every pair of basic subsets of $U$ with the same type is monochromatic. The least integer
$N$ with this property will be denoted by $\mathrm{Ram_B}(k,d,r)$.

Moreover, there exists a primitive recursive function $\chi:\nn^3\to\nn$ belonging to the class $\mathcal{E}^{6}$ of Grzegorczyk's
hierarchy such that 
\begin{equation} \label{e53}
\mathrm{Ram_B}(k,d,r)\mik \chi(k,d,r)
\end{equation}
for every $k,d,r\in\nn$ with $k\meg 2$ and $d,r\meg 1$.
\end{thm}
We have already pointed out in Fact \ref{f25} that Carlson--Simpson trees preserve infima. Also notice that the type of a basic set does not
depend on its top. Taking into account these remarks, it is easy to verify that Proposition \ref{p43} follows from Theorem \ref{t52}. In fact,
Theorem \ref{t52} can be seen as the ``higher-dimensional" extension of Proposition \ref{p43}. 
\begin{proof}[Proof of Theorem \ref{t52}]
It is similar to the proof of Proposition \ref{p43}. In particular, our strategy is to represent all basic subsets of a given Carlson--Simpson
tree as words over an alphabet of the form $[k]\cup \mathcal{L}$ where $\mathcal{L}$ is an appropriately chosen finite set. Once this is done,
the result follows by a straightforward application of Theorem \ref{t14}. We proceed to the details.

Fix $k,d,r\in\nn$ with $k\meg 2$ and $d,r\meg 1$. We set
\begin{equation} \label{e54}
q=d\cdot\sum_{p=2}^{k^d}(k^p-k).
\end{equation}
By Theorem \ref{t14}, it is enough to prove that
\begin{equation} \label{e55}
\mathrm{Ram_B}(k,d,r)\mik \mathrm{FK}(k,q,d,r).
\end{equation}
To this end let $W$ be a Carlson--Simpson tree of $[k]^{<\nn}$ with $\dim(W)\meg \mathrm{FK}(k,q,d,r)$ and fix a coloring $c:\mathrm{B}(W)\to [r]$.
Denote by $\mathbf{w}$ the generating sequence of $W$.

For every $i\in\{0,...,d-1\}$ and every $p\in \{2,...,k^d\}$ let $\mathcal{L}^i_{p}=\mathcal{L}_p\times\{i\}$. Denote by $\mathrm{i}^i_p:
[k]\cup \mathcal{L}_p\to [k]\cup \mathcal{L}^i_p$ the natural bijection defined by $\mathrm{i}^i_p(a)=a$ for every $a\in [k]$ and
$\mathrm{i}^i_p(\lambda)=(\lambda,i)$ for every $\lambda\in \mathcal{L}_{p}$. The map $\mathrm{i}^i_p$ lifts to a bijection
\begin{equation} \label{e56}
\mathrm{I}^i_{p}: \mathrm{W}(k,\mathcal{L}_{p})\to \mathrm{W}(k,\mathcal{L}^i_p)
\end{equation}
defined by $\mathrm{I}^i_{p}\big((s_0,...,s_n)\big)=\big(\mathrm{i}^i_p(s_0),...,\mathrm{i}^i_p(s_n)\big)$. Also let
$\mathrm{R}_p: \mathrm{Fl}_p([k]^{<\nn})\to \mathrm{W}(k,\mathcal{L}_{p})$ be the ``representation" map defined in \S 4.3 and set
\begin{equation} \label{e57}
\mathrm{R}^i_{p}=\mathrm{I}^i_{p}\circ \mathrm{R}_{p}.
\end{equation}
By Fact \ref{f41}, the map $\mathrm{R}^i_{p}:\mathrm{Fl}_p([k]^{<\nn})\to \mathrm{W}(k,\mathcal{L}^i_{p})$ is an injection.

We define
\begin{equation} \label{e58}
\mathcal{L}=\bigcup_{i=0}^{d-1}\bigcup_{p=2}^{k^d}\mathcal{L}^i_{p}
\end{equation}
and we observe that
\begin{equation} \label{e59}
|\mathcal{L}|=\sum_{i=0}^{d-1}\sum_{p=2}^{k^d}|\mathcal{L}^i_{p}|=d\cdot\sum_{p=2}^{k^d}
(k^p-k)=q.
\end{equation}
Also notice that for every $i\in\{0,...,d-1\}$ and every $p\in\{2,...,k^d\}$ we have that $\mathrm{W}(k, \mathcal{L}^i_p)\subseteq
\mathrm{W}(k,\mathcal{L})$. Therefore, the map $\mathrm{R}^i_{p}$ can be seen as an injection from $\mathrm{Fl}_p([k]^{<\nn})$ into
$\mathrm{W}(k,\mathcal{L})$.

Now let $\mathcal{B}$ be the subset of $\mathrm{B}(W)$ defined by
\begin{equation} \label{e510}
B\in\mathcal{B} \Leftrightarrow \text{if } \mathbf{t}(B)=(\tau_0,...,\tau_{\ell-1}), \text{ then } \sum_{i=0}^{\ell-1}|\tau_i|\mik d.
\end{equation}
Observe that for every $B\in\mathcal{B}$ the width $w(B)$ of $B$ is at most $k^d$. We define a map
\begin{equation} \label{e511}
\mathrm{\mathbf{R}}:\mathcal{B} \to \mathrm{W}(k,\mathcal{L})
\end{equation}
by the rule
\begin{equation} \label{e512}
\mathrm{\mathbf{R}}(B)= s^\con \mathrm{R}^0_{p_0}( F_0)^\con\mathrm{R}^1_{p_1}( F_1)^\con...^\con \mathrm{R}^{\ell-1}_{p_{\ell-1}}(F_{\ell-1})
\end{equation}
where $(s, F_0,...,F_{\ell-1})$ is the generating sequence of $B$ and $p_i=|F_i|$ for every $i\in\{0,...,\ell-1\}$. The following properties
follow easily taking into account the relevant definitions.
\begin{enumerate}
\item[(P1)] The map $\mathrm{\mathbf{R}}$ is an injection into $\mathbf{w}(k,\mathcal{L})$.
\item[(P2)] Let $\ell\in\{1,...,d\}$. Also let $\mathbf{p}=(p_0,...,p_{\ell-1})$ and $\mathbf{t}=(\tau_0,...,\tau_{\ell-1})$ where
$p_i\in\{2,...,k^d\}$ and $\tau_i\in \mathcal{T}(\mathcal{L}_{p_i})$ for every $i\in\{0,...,\ell-1\}$. Finally let $V$ be a 
Carlson--Simpson subtree of $W$ generated by some $\mathbf{v}\in\mathrm{Subseq}_d(\mathbf{w})$. Assume that the set
$\mathrm{B}_{\ell,\mathbf{p},\mathbf{t}}(V)$ is nonempty and set
\begin{equation} \label{e513}
\tau=\mathrm{I}^0_{p_0}(\tau_0)^\con...^\con\mathrm{I}^{\ell-1}_{p_{\ell-1}}(\tau_{\ell-1}).
\end{equation}
Then $\tau\in\mathcal{T}(\mathcal{L})$ and $|\tau|=\sum_{i=0}^{\ell-1}|\tau_i|\mik \dim(V)= d$. Moreover, we have 
$\mathrm{B}_{\ell,\mathbf{p},\mathbf{t}}(V)\subseteq \mathcal{B}$ and the map
$\mathrm{\mathbf{R}}: \mathrm{B}_{\ell,\mathbf{p},\mathbf{t}}(V)\to \mathbf{v}(k,\mathcal{L},\tau)$ is a bijection. 
\end{enumerate}
Finally we define $\tilde{c}:\mathbf{w}(k,\mathcal{L})\to [r]$ by
\begin{equation} \label{e514}
\tilde{c}(w)=
\begin{cases}
c(B) & \text{if there exists } B\in\mathcal{B} \text{ with } \mathrm{\mathbf{R}}(B)=w, \\
r    & \text{otherwise.} \\
\end{cases}
\end{equation}
Notice that, by (P1), the coloring $\tilde{c}$ is well-defined. Since $\dim(\mathbf{w})\meg \mathrm{FK}(k,q,d,r)$, there exists
$\mathbf{v}\in\mathrm{Subseq}_d(\mathbf{w})$ such that the set $\mathbf{v}(k,\mathcal{L},\tau)$ is monochromatic with respect to $\tilde{c}$
for every $\tau\in\mathcal{T}(\mathcal{L})$ with $|\tau|\mik d$. Let $V$ be the $d$-dimensional Carlson--Simpson subtree of $W$ generated
by $\mathbf{v}$. Using property (P2), it is easily verified that $V$ is as desired. The proof is completed.
\end{proof}


\section{Embedding subsets of $[k]^{<\nn}$ into basic sets}

\numberwithin{equation}{section}

Our goal in this section is to complete the analysis outlined in \S 4.1. In particular, we shall prove the following embedding result.
\begin{prop} \label{p61}
Let $k\in\nn$ with $k\meg 2$. Also let $W$ be a Carlson--Simpson tree of $[k]^{<\nn}$ and $G\subseteq W$ with $|G|\meg 2$. Then there exists 
a basic set $B$ of dimension and width at most $|G|$, and such that $G\subseteq B\subseteq W$.
\end{prop}
\begin{proof} 
We may assume that $W$ is of the form $[k]^{<d+1}$ for some integer $d\meg 1$. Observe that if $G$ is flat, then we can write $G$ as $s^\con F$
where $s=\wedge G$ and $F$ is a flat set whose infimum is the empty word. In this case, the set $B=\{s\}\cup (s^{\con}F)$ is the desired basic set.
Therefore, in what follows, we may additionally assume that $G$ is not flat.

First we set $L=\{n\in\{0,...,d\}: G\cap [k]^n\neq\varnothing \}$ and $\ell =|L|$. Notice that $2\mik \ell\mik |G|$.
We write the set $L$ in increasing order  as $\{n_1<...<n_{\ell}\}$. For every $i\in\{0,...,\ell-1\}$ let
\begin{equation} \label{e61}
G_{i}=G\cap [k]^{n_{i+1}}
\end{equation}
and 
\begin{equation} \label{e62}
H_i= G_i \cup \bigcup_{j=i+1}^{\ell-1} G_j|n_{i+1}
\end{equation}
where, as in \eqref{e23}, $G_j|n_{i+1}=\{w|n_{i+1}: w\in G_j\}$. Observe that 
\begin{equation} \label{e63}
|G_i|\mik |H_i|\mik \sum_{j=i}^{\ell-1}|G_i|
\end{equation}
for every $i\in\{0,...,\ell-1\}$. Hence,
\begin{enumerate}
\item[(P1)] $|H_0|\mik |G|$ and
\item[(P2)] $\max\{|H_i|:i\in \{1,...,\ell-1\}\}<|G|$.
\end{enumerate}
Let $i\in\{1,...,\ell-1\}$ be arbitrary. For every $w\in [k]^{n_{i+1}}$ denote by $w^*$ the unique element of $[k]^{n_{i+1}-n_{i}}$
such that $w=(w|n_{i})^\con w^*$. We set
\begin{equation} \label{e64}
H_i^*=\{w^*: w\in H_i\}\subseteq [k]^{n_{i+1}-n_{i}}.
\end{equation}
Notice that
\begin{equation} \label{e65}
|H^*_i|\mik |H_i|\stackrel{\text{(P2)}}{<} |G|
\end{equation}
and
\begin{equation} \label{e66}
G_i \subseteq H_i \subseteq H_0 ^\con {H^*_1}^\con...^\con {H_i^*}.
\end{equation}
Next we claim that for every $i\in\{1,...,\ell-1\}$ there exists a flat set $F_i$ such that
\begin{enumerate}
\item[(a)] the infimum $\wedge F_i$ of $F_i$ is the empty word,
\item[(b)] $H_i^*\subseteq F_i$ and
\item[(c)] $|F_i|\mik |H_i^*|+1$.
\end{enumerate}
Indeed, fix $i\in\{1,...,\ell-1\}$. If $\wedge H_i^*$ is the empty word, then we set $F_i=H_i^*$. Otherwise, the set $H_i^*|1$ is a singleton,
and so, we may select $a\in [k]$ with $a\neq H_i^*|1$. Let $\mathbf{a}$ be the unique word of length $n_{i+1}-n_i$ all of whose coordinates are
equal to $a$. We set $F_i=H_i^*\cup\{\mathbf{a}\}$. Clearly $F_i$ is as desired.

We are ready for the final step of the argument. Let $s=\wedge H_0$. Assume, first, that $\{s\}=H_0$. In this case we define
\begin{equation} \label{e67}
B=\{s\}\cup (s^{\con}F_1)\cup ... \cup (s{^\con}F_1^{\con}...^{\con}F_{\ell-1}).
\end{equation}
It is easy to check that $B$ is the desired basic set (in fact, in this case, the dimension of $B$ is at most $|G|-1$). Otherwise, write $H_0$
as $s^{\con}F_0$ where $F_0$ is a flat set whose infimum is the empty word. We define $B$ to be the basic set generated by the sequence
$(s,F_0,F_1,...,F_{\ell-1})$. It is also easily verified that $B$ satisfies all requirements. The proof is completed.
\end{proof}


\section{Correlation on basic sets: one-dimensional case}

\numberwithin{equation}{section}

\subsection{The main result}

We start by introducing some numerical invariants. Specifically for every $k,p\in\nn$ with $k,p\meg 2$ and every $0<\ee\mik 1$ let
\begin{equation} \label{e71}
q_p=q_p(k)=k^p-k,
\end{equation}
\begin{equation} \label{e72}
\Lambda_p=\Lambda_p(k,\ee)=\Lambda(k^{q_p}, \ee^2/4) \stackrel{\eqref{e217}}{=} \lceil 16\ee^{-2}\mathrm{DCS}(k^{q_p},1,\ee^2/16)\rceil
\end{equation}
and 
\begin{equation} \label{e73}
\eta_p(k,\ee)=\eta(k^{q_p},\ee^2/4) \stackrel{\eqref{e218}}{=}\frac{\ee^2}{8|\mathrm{Subtr}_1([k^{q_p}]^{<\Lambda_p})|}.
\end{equation}
Moreover, if $m$ is a positive integer, then we set
\begin{equation} \label{e74}
\mathrm{Cor}^*_{1}(k,m,\ee) = \mathrm{CS}\big(k,m\cdot\Lambda_{k^m}(k,\ee),1,2^k\big).
\end{equation}
We are ready to state the main result of this section.
\begin{prop} \label{p71}
Let $k,m\in \nn$ with $k\meg 2$ and $m\meg 1$, and $0<\ee\mik 1$. Also let $W$ be a Carlson--Simpson tree of $[k]^{<\nn}$ of dimension at least 
$\mathrm{Cor}^*_{1}(k,m,\ee)$. Then for every family $\{A_w: w\in W\}$ of measurable events in a probability space $(\Omega,\Sigma,\mu)$
satisfying $\mu(A_w)\meg \varepsilon$ for every $w\in W$, every $p\in\{2,...,k^m\}$ and every $\tau\in\mathcal{T}_{\mathrm{Fl}}[k,p]$
with $|\tau|\mik m$ there exists a one-dimensional basic set $B\subseteq W$ of width $p$ and type $\tau$ such that
\begin{equation} \label{e75}
\mu\Big(\bigcap_{w\in B}A_w\Big)\meg \eta_p(k,\varepsilon).
\end{equation}
\end{prop}
We emphasize that the lower bound in \eqref{e75} is independent of the length of the type $\tau$. Also we remark that Proposition \ref{p71}
is the first (and crucial) step towards the proof of Theorem \ref{t12}. The proof of Proposition \ref{p71} will be given in \S 7.2. In \S 7.3
we isolate some of its consequences. 

\subsection{Proof of Proposition \ref{p71}}

First we introduce some terminology. For every pair $v,v'\in [k]^{<\nn}$ we say that $v'$ is a \textit{successor} of $v$ if $v$ is a proper
initial segment of $v'$. Notice, in particular, that if $v'$ is a successor of $v$, then $|v|<|v'|$. Moreover, it is easy to see that for every
Carlson--Simpson tree $V$ of $[k]^{<\nn}$ generated by the sequence $\mathbf{v}$ and every $v,v'\in V$ we have that $v'$ is a successor of $v$
if and only if there exist $(u,u_0)\in\mathrm{Subseq}_1(\mathbf{v})$ and $a\in [k]$ such that $v=u$ and $v'=u^\con u_0(a)$. We have the following lemma.
\begin{lem} \label{l72}
Let $k\in\nn$ with $k\meg 2$ and $0<\vartheta<\ee\mik 1$. Also let $N\in\nn$ with $N\meg (\ee^2-\vartheta^2)^{-1}$. Finally let 
$W$ be a Carlson--Simpson tree of $[k]^{<\nn}$ and assume that 
\begin{equation} \label{e76}
\dim(W)\meg \mathrm{CS}\big(k,N,1,2^k\big).
\end{equation} 
Then for every family $\{A_w: w\in W\}$ of measurable events in a probability space $(\Omega,\Sigma, \mu)$ satisfying $\mu(A_w)\meg\ee$
for every $w\in W$ there exists an $N$-dimensional Carlson--Simpson subtree $V$ of $W$ such that for every pair $v,v'\in V$ with $v'$ a successor
of $v$ we have
\begin{equation} \label{e77}
\mu(A_v\cap A_{v'})\meg \vartheta^2.
\end{equation}
\end{lem}
We notice that Lemma \ref{l72} is similar to \cite[Lemma 7.3]{DKT3}. Also we point out that in Proposition \ref{p101} we shall obtain
a significant extension of this result. However, the argument below can serve as an introduction to the proof of Proposition \ref{p101},
and as such, we decided to present it for the benefit of the reader.
\begin{proof}[Proof of Lemma \ref{l72}]
Let $\mathbf{w}$ be the generating sequence of $W$. For every $a\in [k]$ set
\begin{equation} \label{e78}
\mathcal{F}_a=\big\{ (u,u_0)\in \mathrm{Subseq}_1(\mathbf{w}) : \mu(A_u\cap A_{u^{\con}u_0(a)})\meg \vartheta^2\big\}.
\end{equation}
By Theorem \ref{t23} and the estimate in \eqref{e76}, there exists $\mathbf{v}\in\mathrm{Subseq}_N(\mathbf{w})$ such that for every $a\in [k]$
either $\mathrm{Subseq}_1(\mathbf{v})\subseteq \mathcal{F}_a$ or $\mathrm{Subseq}_1(\mathbf{v})\cap\mathcal{F}_a=\varnothing$. Therefore, it is
enough to show that $\mathrm{Subseq}_1(\mathbf{v})\cap\mathcal{F}_a\neq\varnothing$ for every $a\in [k]$. To this end fix $a\in[k]$ and write
$\mathbf{v}$ as $(v,v_0,...,v_{N-1})$. For every $i\in\{1,...,N\}$ we set
\begin{equation} \label{e79}
t_i=v^\con v_0(a)^\con...^\con v_{i-1}(a).
\end{equation}
Since $N\meg (\ee^2-\vartheta^2)^{-1}$, by Lemma \ref{l27}, there exist $i,j\in [N]$ with $i<j$ and such that
$\mu(A_{t_i}\cap A_{t_j})\meg \vartheta^2$. We set
\begin{equation} \label{e710}
u=v^\con v_0(a)^\con...^\con v_{i-1}(a) \ \text{ and } \ u_0=v_i^\con...^\con v_{j-1}.
\end{equation}
Observe that $t_i=u$ and $t_j=u^{\con}u_0(a)$. Hence $(u,u_0)\in\mathrm{Subseq}_1(\mathbf{v})\cap\mathcal{F}_a$ and the proof 
is completed.
\end{proof}
Notice that for every pair $q,n$ of positive integers the sets $([k]^q)^{n}$ and $([k]^{n})^q$ can be naturally identified. Therefore, we may
view a word over $[k]^q$ as a sequence of length $q$ having values in $[k]^n$ for some $n\in\nn$. This is, essentially, the content of the
following definition.
\begin{defn} \label{d43}
Let $k,q\in\nn$ with $k,q\meg 2$. For every positive integer $n$ we define the map
\begin{equation} \label{e711}
\mathrm{I}_{q,n}: ([k]^q)^{n}\to ([k]^{n})^q
\end{equation}
as follows. Fix \ $t=(\mathbf{a}_0,...,\mathbf{a}_{n-1})\in([k]^q)^{n}$ and for every $j\in\{0,...,n-1\}$ and every $i\in\{0,...,q-1\}$ let
$a_{i,j}$ be the $i$-th coordinate of $\mathbf{a}_j$. Next for every $i\in\{0,...,q-1\}$ let $\overline{t}_i=(a_{i,0},...,a_{i,n-1})$ and define
\begin{equation} \label{e712}
\mathrm{I}_{q,n}(t)=(\overline{t}_0,...,\overline{t}_{q-1}).
\end{equation}
Also let
\begin{equation} \label{e713}
\mathrm{I}_q:([k]^q)^{<\nn}\to \{\varnothing\} \cup \bigcup_{n=1}^\infty([k]^{n})^q
\end{equation}
be defined by $\mathrm{I}_q(\varnothing)=\varnothing$ and 
\begin{equation} \label{e714}
\mathrm{I}_q(t)=\mathrm{I}_{q,|t|}(t)
\end{equation}
for every  $t\in ([k]^q)^{<\nn}$ with $|t|\meg 1$.
\end{defn}
We will need the following elementary fact.
\begin{fact} \label{f74}
Let $k,q\in\nn$ with $k,q\meg 2$. Also let $s$ be a variable word over $[k]^q$. Then there exists a unique sequence $(s_0,...,s_{q-1})$
such that $s_i$ is a variable word over $k$ with $|s_i|=|s|$ for every $i\in\{0,...,q-1\}$ and satisfying
\begin{equation} \label{e715}
\mathrm{I}_{q}\big(s(\mathbf{a})\big)=\big(s_0(a_0),...,s_{q-1}(a_{q-1})\big)
\end{equation}
for every $\mathbf{a}=(a_0,...,a_{q-1})\in [k]^q$. Moreover, if $s$ is a left variable word, then so is $s_i$ for every $i\in\{0,...,q-1\}$.
\end{fact}
We are about to introduce one more definition. To motivate the reader, let us first notice that every word over $[k]^q$ of length $n$ can be
represented as a word over $k$ of length $qn$, just by concatenating its coordinates. We will need, however, a rather different representation.
Its main properties are described in Lemma \ref{l77} below. Specifically, the representation is designed so that the image of every Carlson--Simpson
line of $([k]^q)^{<\nn}$ contains a basic set of $[k]^{<\nn}$ of a given type $\tau$. 
\begin{defn} \label{d75}
Let $k,p\in\nn$ with $k,p\meg 2$. Define $q_p$ and $\mathcal{L}_p$ as in \eqref{e71} and \eqref{e43} respectively. Notice that 
$|\mathcal{L}_p|=q_p$. Write the set $\mathcal{L}_p$ in lexicographical increasing order as
\begin{equation} \label{e716}
\mathcal{L}_p=\{\lambda_0<_{\mathrm{lex}} \cdots <_{\mathrm{lex}}\lambda_{q_p-1}\}.
\end{equation} 
Let $\tau\in\mathcal{T}(\mathcal{L}_p)$ and set $\ell=|\tau|$. Consider the unique sequence $(i_0,...,i_{\ell-1})$ in $\{0,...,q_p-1\}$
such that
\begin{equation} \label{e717}
\tau=(\lambda_{i_0},...,\lambda_{i_{\ell-1}}).
\end{equation}
We define two maps
\begin{equation} \label{e718}
\Phi_{p,\tau}:([k]^{q_p})^{<\nn}\to [k]^{<\nn} \ \text{ and } \ \Phi^0_{p,\tau}:([k]^{q_p})^{<\nn}\to [k]^{<\nn}
\end{equation}
as follows. First set $\Phi_{p,\tau}(\varnothing)=\Phi_{p,\tau}^0(\varnothing)=\varnothing$. Otherwise, let $t\in([k]^{q_p})^{<\nn}$
with $|t|\meg 1$ and write $\mathrm{I}_{q_p}(t)$ as $(\overline{t}_0,...,\overline{t}_{q_p-1})$. We set
\begin{equation} \label{e719}
\Phi_{p,\tau}(t)= (\overline{t}_{i_0})^\con...^\con(\overline{t}_{i_{\ell-1}}) \ \text{ and } \ 
\Phi^0_{p,\tau}(t)= \overline{t}_{i_0}.
\end{equation}
\end{defn}
We isolate, for future use, some properties of the maps $\Phi_{p,\tau}$ and $\Phi^0_{p,\tau}$. All of them are straightforward consequences 
of the relevant definition.
\begin{fact} \label{f76}
Let $k,p\in\nn$ with $k,p\meg 2$ and $\tau\in\mathcal{T}(\mathcal{L}_p)$. 
\begin{enumerate}
\item[(a)] For every $t\in ([k]^{q_p})^{<\nn}$ the word $\Phi^0_{p,\tau}(t)$ is an initial segment of $\Phi_{p,\tau}(t)$. Moreover,
if $|\tau|\meg 2$ and $|t|\meg 1$, then $\Phi_{p,\tau}(t)$ is a successor of $\Phi^0_{p,\tau}(t)$.
\item[(b)] For every $t\in ([k]^{q_p})^{<\nn}$ we have $|\Phi_{p,\tau}(t)|=|t|\cdot|\tau|$.
\item[(c)] For every $t,s\in([k]^{q_p})^{<\nn}$ with $|t|,|s|\meg 1$ if
$\Phi_{p,\tau}(t)=(\overline{t}_{i_0})^\con...^\con(\overline{t}_{i_{\ell-1}})$ and 
$\Phi_{p,\tau}(s)= (\overline{s}_{i_0})^\con...^\con(\overline{s}_{i_{\ell-1}})$, then 
\begin{equation} \label{e720}
\Phi_{p,\tau}(t^\con s)=(\overline{t}_{i_0})^\con(\overline{s}_{i_0})^\con ... ^\con(\overline{t}_{i_{\ell-1}})^\con
(\overline{s}_{i_{\ell-1}}).
\end{equation}
\end{enumerate}
\end{fact} 
We proceed with the following lemma.
\begin{lem} \label{l77}
Let $k,p\in\nn$ with $k,p\meg 2$ and $\tau\in\mathcal{T}_{\mathrm{Fl}}[k,p]$. Also let $q_p$ be as in \eqref{e71}.
\begin{enumerate}
\item[(a)] Let $s$ be a variable word over $[k]^{q_p}$ and set $n=|s|$. Then the set 
\begin{equation} \label{e721}
\big\{\Phi_{p,\tau}\big(s(\mathbf{a})\big): \mathbf{a}\in [k]^{q_p}\big\}
\end{equation} contains a flat set $F\subseteq [k]^{n\cdot |\tau|}$ of cardinality $p$ and type $\tau$.
\item[(b)] Let $t$ be a word over $[k]^{q_p}$ and $s$ be a left variable word over $[k]^{q_p}$ and set $n=|t|+|s|$. Then the set
\begin{equation} \label{e722}
\{\Phi^0_{p,\tau}(t)\}\cup\big\{\Phi_{p,\tau}\big(t^\con s(\mathbf{a})\big): \mathbf{a}\in [k]^{q_p}\big\}
\end{equation} 
contains a basic set $B\subseteq [k]^{<n\cdot |\tau|+1}$ of dimension $1$, width $p$ and type $\tau$. Moreover, the top of $B$ is the word $\Phi^0_{p,\tau}(t)$.
\end{enumerate}
\end{lem}
\begin{proof}
First we need to do some preparatory work. As in Definition \ref{d75}, write the set $\mathcal{L}_p$ in lexicographical increasing order
as $\{\lambda_0<_{\mathrm{lex}}...<_{\mathrm{lex}}\lambda_{q_p-1}\}$. Also let $\ell=|\tau|$ and $(i_0,...,i_{\ell-1})$
be the unique sequence in $\{0,...,q_p-1\}$ such that $\tau=(\lambda_{i_0},...,\lambda_{i_{\ell-1}})$. 
For every $i\in\{0,...,q_p-1\}$ and every $m\in\{0,...,p-1\}$ let $l_{i,m}\in [k]$ be the $m$-th coordinate of $\lambda_i$.
We set
\begin{equation} \label{e723}
\mathbf{b}_m=(l_{0,m},...,l_{q_p-1,m})\in [k]^{q_p}
\end{equation}
for every $m\in\{0,...,p-1\}$.
\medskip

\noindent (a) Let
\begin{equation} \label{e724}
F=\big\{\Phi_{p,\tau}\big(s(\mathbf{b}_m)\big): 0\mik m\mik p-1\big\}.
\end{equation}
We claim that $F$ is as desired. Indeed, by Fact \ref{f76}, we have that $F\subseteq [k]^{n\cdot|\tau|}$. Hence $F$ is flat. It is easily
seen that the cardinality of $F$ is $p$. Therefore, it is enough to show that $F$ is of type $\tau$. To this end let $(s_0,...,s_{q_p-1})$
be the sequence of variable words over $k$ obtained by Fact \ref{f74} for the variable word $s$. By \eqref{e715} and \eqref{e719}, we see that
\begin{equation} \label{e725}
\Phi_{p,\tau}\big(s(\mathbf{b}_m)\big)=s_{i_0}(l_{i_0,m})^\con ...^\con s_{i_{\ell-1}}(l_{i_{\ell-1},m})
\end{equation}
for every $m\in\{0,...,p-1\}$. Noticing that $\lambda_{i_j}=\left\langle l_{i_j,m}: m\in\{0,...,p-1\}\right\rangle$ for every
$j\in\{0,...,\ell-1\}$, by \eqref{e725}, we conclude that $F$ has type $(\lambda_{i_0},...,\lambda_{i_{\ell-1}})=\tau$. 
\medskip

\noindent (b) We will give the proof under the additional assumption that $|t|\meg 1$. If $t$ is the empty word, then the proof 
is similar (in fact, it is simpler). First let
\begin{equation} \label{e726}
\mathrm{I}_{q_p}(t)=(\overline{t}_0,...,\overline{t}_{q_p-1})
\end{equation}
and $(s_0,...,s_{q_p-1})$ be the sequence of variable words over $k$ obtained by Fact \ref{f74} for the left variable word $s$. Notice,
in particular, that $s_i$ is a left variable word for every $i\in\{0,...,q_p-1\}$. Also observe that $\Phi^0_{p,\tau}(t)= \overline{t}_{i_0}$.
Moreover, by Fact \ref{f76} and \eqref{e726}, we have 
\begin{equation} \label{e727}
\Phi_{p,\tau}\big(t^\con s(\mathbf{b}_m)\big)=(\overline{t}_{i_0})^\con s_{i_0}(l_{i_0,m})^\con ...^\con
(\overline{t}_{i_{\ell-1}})^\con s_{i_{\ell-1}}(l_{i_{\ell-1},m})
\end{equation}
for every $m\in\{0,...,\ell-1\}$. We set
\begin{equation} \label{e728}
F_0=\big\{ s_{i_0}(l_{i_0,m})^\con(\overline{t}_{i_1})^\con ... ^\con(\overline{t}_{i_{\ell-1}})^\con s_{i_{\ell-1}}(l_{i_{\ell-1},m}): 
0\mik m\mik p-1\big\}.
\end{equation}
Arguing as in the first part of the proof, we see that $F_0$ is a flat set which is contained in $[k]^{n\cdot |\tau|-|t|}$ and is of cardinality
$p$ and type $\tau$. Since $(s_0,...,s_{q_p-1})$ consists of left variable words, it is easy to verify that $\wedge F_0$ is the empty word. 
Therefore, setting 
\begin{equation} \label{e729}
B=\{\overline{t}_{i_0}\} \cup \big( ({\overline{t}_{i_0}})^\con F_0\big),
\end{equation}
we see that $B$ is as desired. The proof is thus completed. 
\end{proof}
The following lemma is the final step towards the proof of Proposition \ref{p71}.
\begin{lem} \label{l78}
Let $k,p,\ell\in\nn$ with $k,p\meg 2$ and $\ell\meg 1$. Also let $0<\ee\mik 1$ and define $\Lambda_p=\Lambda_p(k,\ee)$ as in \eqref{e72}.
Finally let $W$ be a Carlson--Simpson tree of $[k]^{<\nn}$ and assume that 
\begin{equation} \label{e730}
\dim(W)\meg\mathrm{CS}\big(k,\ell\cdot \Lambda_p,1,2^k\big).
\end{equation}
Then for every family $\{A_w: w\in W\}$ of measurable events in a probability space $(\Omega,\Sigma,\mu)$ satisfying $\mu(A_w)\meg\ee$ for
every $w\in W$ and every $\tau\in \mathcal{T}_{\mathrm{Fl}}[k,p]$ with $|\tau|=\ell$ there exists an one-dimensional basic set $B\subseteq W$
of width $p$ and type $\tau$ such that
\begin{equation} \label{e731}
\mu \Big(\bigcap_{w\in B}A_w \Big)\meg \eta_p(k,\varepsilon)
\end{equation}
where $\eta_p(k,\varepsilon)$ is as in \eqref{e73}.
\end{lem}
\begin{proof} 
Notice first that, by Lemma \ref{l72} applied for ``$\vartheta=\ee/2$" and ``$d=\ell\cdot\Lambda_p$", there exists a Carlson--Simpson subtree
$V$ of $W$ with $\dim(V)=\ell\cdot \Lambda_p$ such that
\begin{equation} \label{e732}
\mu(A_v\cap A_{v'})\meg \ee^2/4
\end{equation}
for every $v,v'\in V$ with $v'$ successor of $v$. 

Next observe that there is a natural bijection between the alphabets $[k]^{q_p}$ and $[k^{q_p}]$. This bijection extends, of course, 
to a bijection between $([k]^{q_p})^{<\nn}$ and $[k^{q_p}]^{<\nn}$. Keeping in mind these remarks, in what follows we will identify
words (respectively, left variable words) over $[k]^{q_p}$ with words (respectively, left variable words) over $[k^{q_p}]$. We will
also identify $V$ with $[k]^{<\ell\cdot\Lambda_p+1}$ via the canonical isomorphism $\mathrm{I}_V$.

Fix $\tau\in \mathcal{T}_{\mathrm{Fl}}[k,p]$ with $|\tau|=\ell$. Let $q_p$ be as in \eqref{e71}. For every $t\in ([k]^{q_p})^{<\Lambda_p}$ set
\begin{equation} \label{e733}
\tilde{A}_t=A_{\Phi^0_{p,\tau}(t)}\cap A_{\Phi_{p,\tau}(t)}
\end{equation}
where $\Phi_{p,\tau}$ and $\Phi^0_{p,\tau}$ are as in Definition \ref{d75}. By Fact \ref{f76}, we have
\begin{equation} \label{e734}
\Phi^0_{p,\tau}\big(([k]^{q_p})^{<\Lambda_p}\big)\subseteq [k]^{<\ell\cdot \Lambda_p+1} \ \text{ and } \
\Phi_{p,\tau}\big(([k]^{q_p})^{<\Lambda_p}\big)\subseteq [k]^{<\ell\cdot \Lambda_p+1}.
\end{equation}
Hence $\tilde{A}_t$ is well-defined. Invoking Fact \ref{f76} once again and the estimate in \eqref{e732}, we see that
$\mu(\tilde{A}_t)\meg \varepsilon^2/4$ for every $t\in ([k]^{q_p})^{<\Lambda_p}$. By the choice of $\Lambda_p$ in \eqref{e72} and Lemma
\ref{l29} applied for ``$k=k^{q_p}$" and ``$\delta=\ee^2/4$", there exist a word $t$ over $[k]^{q_p}$ and a left variable word $t_0$
over $[k]^{q_p}$ such that, setting
\begin{equation} \label{e735}
S=\{t\}\cup \big\{t^{\con}t_0(\mathbf{a}): \mathbf{a}\in [k]^{q_p}\big\},
\end{equation}
we have 
\begin{equation} \label{e736}
\mu \Big(\bigcap_{t\in S}\tilde{A}_t\Big)\meg \eta(k^{q_p},\ee^2/4)\stackrel{\eqref{e73}}{=}\eta_p(k,\epsilon).
\end{equation}
Notice that
\begin{align} \label{e737}
\bigcap_{t\in S}\tilde{A}_t & = A_{\Phi^0_{p,\tau}(t)}\cap A_{\Phi_{p,\tau}(t)}\cap \bigcap_{\mathbf{a}\in [k]^{q_p}}
A_{\Phi^0_{p,\tau}\big(t^\con t_0(\mathbf{a})\big)}\cap A_{\Phi_{p,\tau}\big(t^\con t_0(\mathbf{a})\big)} \\
& \subseteq A_{\Phi^0_{p,\tau}(t)}\cap \bigcap_{\mathbf{a}\in [k]^{q_p}}A_{\Phi_{p,\tau}\big(t^\con t_0(\mathbf{a})\big)} \notag.
\end{align}
By Lemma \ref{l77}, there exists a one-dimensional basic set $B$ of width $p$ and type $\tau$ such that
\begin{equation} \label{e738}
B\subseteq \big\{\Phi^0_{p,\tau}(t)\big\} \cup \big\{\Phi_{p,\tau}\big( t^\con t_0(\mathbf{a})\big): \mathbf{a}\in [k]^{q_p}\big\}.
\end{equation}
By \eqref{e736}, \eqref{e737} and \eqref{e738}, we conclude that
\begin{equation} \label{e739}
\mu \Big(\bigcap_{w\in B}A_w\Big)\meg \eta_p(k,\varepsilon)
\end{equation}
and the proof is completed. 
\end{proof}
We are ready to complete the proof of Proposition \ref{p71}.
\begin{proof}[Proof of Proposition \ref{p71}]
Let $\{A_w: w\in W\}$ be a family of measurable events in a probability space $(\Omega,\Sigma,\mu)$ satisfying $\mu(A_w)\meg \ee$
for every $w\in W$. Also let $p\in \{2,...,k^m\}$ and $\tau\in\mathcal{T}_{\mathrm{Fl}}[k,p]$ with $|\tau|=\ell\mik m$. By
\eqref{e72}, we see that $\Lambda_p(k,\ee)\mik \Lambda_{k^m}(k,\ee)$ and so
\begin{equation} \label{e740}
\dim(W)\meg \mathrm{Cor}^*_1(k,m,\ee) \stackrel{\eqref{e74}}{\meg}\mathrm{CS}\big(k,\ell\cdot\Lambda_{p},1,2^k\big).
\end{equation}
Thus, invoking Lemma \ref{l78}, the result follows.
\end{proof}

\subsection{Consequences}

We will need the following consequence of Proposition \ref{p71}.
\begin{cor} \label{c79}
Let $k\in\nn$ with $k\meg 2$ and $0<\varrho\mik 1$. Also let $m,\ell\in\nn$ with $1\mik \ell<m$ and $V$ be a Carlson--Simpson tree
of $[k]^{<\nn}$ with
\begin{equation} \label{e741}
\dim(V)\meg\mathrm{Cor}^*_{1}(k,m,\varrho)+m.
\end{equation}
Finally let $\mathbf{p}=(p_0,...,p_{\ell-1})$ and $\mathbf{t}=(\tau_0,...,\tau_{\ell-1})$ where $p_i\in\nn$ with $p_i\meg 2$ and 
$\tau_i\in\mathcal{T}_{\mathrm{Fl}}[k,p_i]$ for every $i\in\{0,...,\ell-1\}$, and assume that $\sum_{i=0}^{\ell-1}|\tau_i|\mik m$.
Assume, moreover, that $\{A_v:v\in V\}$ is a family of measurable events in a probability space $(\Omega,\Sigma,\mu)$ such that
for every $B\in\mathrm{B}_{\ell,\mathbf{p},\mathbf{t}}(V)$ we have
\begin{equation} \label{e742}
\mu\Big( \bigcap_{v\in B} A_v\Big) \meg \varrho.
\end{equation} 
Then for every $p\in \{2,...,k^m\}$ and every $\tau\in\mathcal{T}_{\mathrm{Fl}}[k,p]$ with $|\tau|\mik m$ there exists an $(\ell+1)$-dimensional
basic set $B'\subseteq V$ with $\mathbf{p}(B')=(p, p_0,...,p_{\ell-1})$ and $\mathbf{t}(B')=(\tau,\tau_0,...,\tau_{\ell-1})$, and such that
\begin{equation} \label{e743}
\mu\Big( \bigcap_{v\in B'} A_v\Big) \meg \eta_p(k,\varrho)
\end{equation}
where $\eta_p(k,\varrho)$ is as in \eqref{e73}.
\end{cor}
\begin{proof}
Set $d=\dim(V)$ and identify $V$ with $[k]^{<d+1}$ via the canonical isomorphism $\mathrm{I}_{V}$. Also fix a basic set $B_1$ of
$[k]^{<m+1}$ with top the empty word, $\dim(B_1)=\ell$, $\mathbf{p}(B_1)=\mathbf{p}$ and $\mathbf{t}(B_1)=\mathbf{t}$. We set
\begin{equation} \label{e744}
d_0=\mathrm{Cor}^*_{1}(k,m,\varrho)
\end{equation}
and for every $s\in [k]^{<d_0+1}$ let
\begin{equation} \label{e745}
B_s=s^\con B_1 \ \text{ and } \ \tilde{A}_s=\bigcap_{v\in B_s} A_v.
\end{equation}
Let $s\in [k]^{<d_0+1}$ be arbitrary. Since $d\meg d_0+m$ we see that $B_s$ is a basic set of $[k]^{<d+1}$ with top the word $s$,
$\dim(B_s)=\ell$, $\mathbf{p}(B_s)=\mathbf{p}$ and $\mathbf{t}(B_s)=\mathbf{t}$. In particular, we have that 
$B_s\in\mathrm{B}_{\ell,\mathbf{p},\mathbf{t}}([k]^{<d+1})$. Hence, by our assumptions, we conclude that 
\begin{equation} \label{e746}
\mu(\tilde{A}_s)\meg \varrho.
\end{equation}
Now let $p\in \{2,...,k^m\}$ and $\tau\in\mathcal{T}_{f}[k,p]$ with $|\tau|\mik m$. By \eqref{e744} and \eqref{e746}, we may apply
Proposition \ref{p71} for the family $\{\tilde{A}_s: s\in [k]^{<d_0+1}\}$, the integer $p$ and the type $\tau$. Therefore, there exists
an one-dimensional basic set $B_0\subseteq [k]^{<d_0+1}$ of width $p$ and type $\tau$, and such that
\begin{equation} \label{e747}
\mu\Big( \bigcap_{s\in B_0}\tilde{A}_s\Big) \meg \eta_p(k,\varrho).
\end{equation}
Let $(s_0,F_0)$ be the generating sequence of $B_0$ and define
\begin{equation} \label{e748}
B'=\{s_0\}\cup (s_0 ^\con F_0) \cup (s^\con F_0^\con B_1).
\end{equation}
Since the top of $B_1$ is the empty word, we see that $B'$ is an $(\ell+1)$-dimensional basic set of $[k]^{<d+1}$
with $\mathbf{p}(B')=(p,p_0,...,p_{\ell-1})$ and $\mathbf{t}(B')=(\tau,\tau_0,...,\tau_{\ell-1})$. Also notice that
$B'\subseteq B_0^\con B_1$. Therefore,
\begin{equation} \label{e749}
\bigcap_{s\in B_0} \tilde{A}_s \stackrel{\eqref{e745}}{=}\bigcap_{v\in B_0^\con B_1} A_v \subseteq \bigcap_{v\in B'} A_v
\end{equation}
and so
\begin{equation} \label{e750}
\mu\Big( \bigcap_{v\in B'} A_v\Big) \meg \mu\Big( \bigcap_{s\in B_0}\tilde{A}_s\Big)
\stackrel{\eqref{e747}}{\meg}\eta_p(k,\varrho).
\end{equation}
The proof is completed.
\end{proof}


\section{Correlation on basic sets: higher-dimensional case}

\numberwithin{equation}{section}

Our goal in this section is to obtain a higher-dimensional extension of Proposition \ref{p71}. This extension is the final step of the proof
of Theorem \ref{t12}. To state it we need to introduce some numerical invariants. 

Let $k\in\nn$ with $k\meg 2$ and $0<\ee\mik 1$. Also let $\mathbf{p}=(p_0,...,p_{\ell-1})$
be a nonempty finite sequence in $\nn$ with $p_i\meg 2$ for every $i\in\{0,...,\ell-1\}$. We define a sequence $(\eta_i)_{i=0}^{\ell-1}$ of
positive reals recursively by the rule
\begin{equation} \label{e81}
\begin{cases}  \eta_0=\eta_{p_{\ell-1}}(k,\ee), \\ \eta_{i+1}=\eta_{p_{\ell-i-2}}(k,\eta_i) \end{cases}
\end{equation}
where $\eta_{p_{\ell-1}}(k,\ee)$ and $\eta_{p_{\ell-i-2}}(k,\eta_i)$ are as in \eqref{e73}. We set
\begin{equation} \label{e82}
\eta_{\mathbf{p}}(k,\ee)=\eta_{\ell-1}.
\end{equation}
Also let $m\in\nn$ with $m\meg 1$ and define
\begin{equation} \label{e83}
\boldsymbol{m}(k)=(\underbrace{k^m,..., k^m}_{m\mathrm{-times}})
\end{equation}
and 
\begin{equation} \label{e84}
\mathrm{Cor}(k,m,\varepsilon) = \mathrm{Ram}_{\mathrm{B}}\Big(k,\mathrm{Cor}^*_{1}\big(k,m,\eta_{\boldsymbol{m}(\!k)}(k,\ee)\big)+m,2\Big)
\end{equation}
where $\mathrm{Cor}^*_{1}\big(k,m,\eta_{\boldsymbol{m}(\!k)}(k,\ee)\big)$ is as in \eqref{e74}. We have the following proposition.
\begin{prop} \label{p81}
Let $k,m\in \nn$ with $k\meg 2$ and $m\meg 1$, and $0<\ee\mik 1$. Also let $W$ be a Carlson--Simpson tree of $[k]^{<\nn}$ of dimension at least 
$\mathrm{Cor}(k,m,\ee)$. Then for every family $\{A_w: w\in W\}$ of measurable events in a probability space $(\Omega,\Sigma,\mu)$ satisfying
$\mu(A_w)\meg \varepsilon$ for every $w\in W$ there exists a Carlson--Simpson subtree $U$ of $W$ with $\dim(U)=m$ such that for every basic
set $B$ of $U$ we have 
\begin{equation} \label{e85}
\mu \Big(\bigcap_{w\in B} A_w \Big)\meg \eta_{\mathbf{p}(B)}(k,\varepsilon).
\end{equation}
\end{prop}
\begin{proof}
We fix a Carlson--Simpson tree $W$ of $[k]^{<\nn}$ of dimension at least $\mathrm{Cor}(k,m,\ee)$ and a family $\{A_w:w\in W\}$ of measurable
events in a probability space $(\Omega,\Sigma,\mu)$ satisfying $\mu(A_w)\meg\ee$ for every $w\in W$. We set
\begin{equation} \label{e86}
\mathcal{G}=\Big\{ B\in \mathrm{B}(W): \mu\Big(\bigcap_{w\in B} A_w \Big) \meg \eta_{\mathbf{p}(B)}(k,\varepsilon)\Big\}.
\end{equation}
By Theorem \ref{t52} and \eqref{e84}, there exists a Carlson--Simpson subtree $V$ of $W$ with 
\begin{equation} \label{e87}
\dim(V)=\mathrm{Cor}^*_{1}\big(k,m,\eta_{\boldsymbol{m}(\!k)}(k,\ee)\big)+m
\end{equation}
and satisfying the following property. For every pair $B_1, B_2$ of basic sets of $V$ with the same type we have that $B_1\in\mathcal{G}$
if and only if $B_2\in\mathcal{G}$.
\begin{claim} \label{c82}
Let $\ell\in [m]$. Also let $\mathbf{p}=(p_0,...,p_{\ell-1})$ and $\mathbf{t}=(\tau_0,...,\tau_{\ell-1})$ where $p_i\in\nn$ with $p_i\meg 2$
and $\tau_i\in\mathcal{T}_{\mathrm{Fl}}[k,p_i]$ for every $i\in\{0,...,\ell-1\}$. Assume that $\sum_{i=0}^{\ell-1}|\tau_i|\mik m$. Then we
have that $\mathrm{B}_{\ell,\mathbf{p},\mathbf{t}}(V)\subseteq\mathcal{G}$.
\end{claim}
Notice that, by Claim \ref{c82}, any $m$-dimensional Carlson--Simpson subtree $U$ of $V$ satisfies the requirements of the proposition.
Therefore, the proof will be completed once we prove Claim \ref{c82}.

To this end we will proceed by induction on $\ell$. The case ``$\ell=1$" follows immediately form Proposition \ref{p71} since
\begin{equation} \label{e88}
\dim(V)\meg \mathrm{Cor}^*_{1}\big(k,m,\eta_{\boldsymbol{m}(\!k)}(k,\ee)\big) \meg \mathrm{Cor}^*_1(k,m,\ee).
\end{equation}
So let $\ell\in [m-1]$ and assume that the claim has been proved up to $\ell$. We fix two sequences $\mathbf{p}=(p_0,...,p_\ell)$
and $\mathbf{t}=(\tau_0,...,\tau_\ell)$ where $p_i\in\nn$ with $p_i\meg 2$ and $\tau_i\in\mathcal{T}_{\mathrm{Fl}}[k,p_i]$ for every
$i\in\{0,...,\ell\}$, and such that $\sum_{i=0}^{\ell}|\tau_i|\mik m$. We set
\begin{equation} \label{e89}
\mathbf{p}'=(p_1,..., p_\ell) \ \text{ and } \ \mathbf{t}'=(\tau_1,...,\tau_\ell).
\end{equation}
By our inductive assumptions, we see that $\mathrm{B}_{\ell,\mathbf{p}',\mathbf{t}'}(V)\subseteq \mathcal{G}$. This is equivalent
to saying that
\begin{equation} \label{e810}
\mu\Big( \bigcap_{w\in B'} A_w\Big) \meg \eta_{\mathbf{p}'}(k,\varepsilon)
\end{equation}
for every $B'\in\mathrm{B}_{\ell,\mathbf{p}',\mathbf{t}'}(V)$. Observe that $\eta_{\boldsymbol{m}(\!k)}(k,\ee)\mik \eta_{\mathbf{p}'}(k,\ee)$.
Hence, by \eqref{e87},  
\begin{equation} \label{e811}
\dim(V) \meg \mathrm{Cor}^*_1\big(k,m,\eta_{\mathbf{p}'}(k,\ee)\big)+m.
\end{equation}
By \eqref{e810} and \eqref{e811}, we may apply Corollary \ref{c79} for ``$\varrho=\eta_{\mathbf{p}'}(k,\varepsilon)$", ``$p=p_0$" and
``$\tau=\tau_0$". In particular, there exists a basic set $B$ of $V$ with $\mathbf{p}(B)=(p_0,p_1,...,p_\ell)$ and
$\mathbf{t}(B)=(\tau_0,\tau_1,...,\tau_{\ell})$, and such that
\begin{equation} \label{e812}
\mu \Big( \bigcap_{w\in B} A_w \Big)\meg \eta_{p_0}\big(k,\eta_{\mathbf{p}'}(k,\ee)\big)
\stackrel{\eqref{e82}} {=}\eta_{\mathbf{p}}(k,\varepsilon)=\eta_{\mathbf{p}(B)}(k,\varepsilon).
\end{equation}
It follows that $B\in \mathrm{B}_{\ell+1,\mathbf{p},\mathbf{t}}(V)\cap \mathcal{G}$ and so $\mathrm{B}_{\ell+1,\mathbf{p},\mathbf{t}}(V)
\subseteq \mathcal{G}$ by the choice of $V$. This completes the proof of Claim \ref{c82} and, as we have already indicated, the entire proof
is completed. 
\end{proof}


\section{Proofs of Theorem \ref{t12} and Corollary \ref{c13}}

\numberwithin{equation}{section}

We start with the proof of Theorem \ref{t12}.

\begin{proof}[Proof of Theorem \ref{t12}]
Let $k\in\nn$ with $k\meg 2$ and $0<\ee\mik 1$. We set 
\begin{equation} \label{e91}
\theta(k,\ee,1)=\ee.
\end{equation}
If $n\in\nn$ with $n\meg 2$, then let
\begin{equation} \label{e92}
\boldsymbol{n}=(\underbrace{n,...,n}_{n\mathrm{-times}})
\end{equation}
and define 
\begin{equation} \label{e93}
\theta(k,\ee,n)=\eta_{\boldsymbol{n}}(k,\ee)
\end{equation}
where $\eta_{\boldsymbol{n}}(k,\ee)$ is as in \eqref{e82}. Finally, for every positive integer $m$ let $\mathrm{Cor}(k,m,\ee)$ be as in
\eqref{e84}. We claim that with these choices the result follows. 

Indeed, fix a Carlson--Simpson tree $T$ of $[k]^{<\nn}$ of dimension at least $\mathrm{Cor}(k,m,\ee)$ and a family $\{A_t:t\in T\}$
of measurable events in a probability space $(\Omega,\Sigma,\mu)$ satisfying $\mu(A_t)\meg\ee$ for every $t\in T$. By Proposition
\ref{p81}, there exists a Carlson--Simpson subtree $S$ of $T$ of dimension $m$ such that
\begin{equation} \label{e94}
\mu \Big( \bigcap_{t\in B} A_t\Big) \meg \eta_{\mathbf{p}(B)}(k,\varepsilon)
\end{equation}
for every basic set $B$ of $S$. We will show that $S$ is as desired. To this end let $F$ be a nonempty subset of $S$ and set $n=|F|$.
If $n=1$, then the estimate in \eqref{e13} is automatically satisfied by the choice of $\theta(k,\ee,1)$ in \eqref{e91} and our 
assumptions. Otherwise, by Proposition \ref{p61}, there exists a basic set $B$ of width and dimension at most $n$ and such that
$F\subseteq B\subseteq S$. In particular, if $\mathbf{p}(B)=(p_0,...,p_{\ell-1})$, then $\ell\in [n]$ and $p_i\in \{2,...,n\}$ for
every $i\in \{0,...,\ell-1\}$. This is easily seen to imply that $\eta_{\mathbf{p}(B)}(k,\varepsilon)\meg \eta_{\boldsymbol{n}}(k,\ee)$.
Therefore,
\begin{equation} \label{e95}
\mu \Big( \bigcap_{t\in F} A_t\Big) \meg \mu \Big( \bigcap_{t\in B} A_t\Big) \stackrel{\eqref{e94}}{\meg} \eta_{\mathbf{p}(B)}(k,\varepsilon)
\meg \eta_{\boldsymbol{n}}(k,\ee) \stackrel{\eqref{e93}}{=} \theta(k,\ee,|F|).
\end{equation}
The proof of Theorem \ref{t12} is thus completed. 
\end{proof}
We proceed to the proof of Corollary \ref{c13}.
\begin{proof}[Proof of Corollary \ref{c13}]
As we have already indicated in the introduction, we will reduce the proof to Theorem \ref{t12}. We will argue as in \cite[Proposition 11.13]{DKT3}. 
Specifically, fix an integer $N\meg \mathrm{Cor}(k,\ee,m)$ and a family $\{A_w:w\in [k]^N\}$ of measurable events in a probability
space $(\Omega,\Sigma,\mu)$ satisfying $\mu(A_w)\meg \ee$ for every $w\in [k]^N$. For every $i\in \{0,...,N\}$ we fix a word $c_i\in [k]^{N-i}$,
and for every $t\in [k]^i$ let $\tilde{t}=t^{\con}c_i\in [k]^N$ and define
\begin{equation} \label{e96}
\tilde{A}_t=A_{\tilde{t}}.
\end{equation}
By Theorem \ref{t12}, there exists an $m$-dimensional Carlson--Simpson subtree $S$ of $[k]^{<N+1}$ such that for every nonempty
$F\subseteq S$ we have 
\begin{equation} \label{e97}
\mu\Big( \bigcap_{t\in F} \tilde{A}_t\Big) \meg \theta(k,\ee,|F|).
\end{equation}
We set $V=\{\tilde{t}: t\in S(m)\}$. Clearly $V$ is as desired.
\end{proof} 


\section{Estimating the constant $\theta(k,\ee,2)$}

\numberwithin{equation}{section}

This section is devoted to the proof of the following result.
\begin{prop} \label{p101}
Let $k,m\in\nn$ with $k\meg 2$ and $m\meg 1$, and $0<\vartheta<\ee\mik 1$. Then there exists a positive integer 
$\mathrm{Cor}_{2}(k, m, \vartheta,\varepsilon)$ with the following property. If $W$ is a Carlson--Simpson tree of $[k]^{<\nn}$ of
dimension at least $\mathrm{Cor}_{2}(k, m, \vartheta,\varepsilon)$ and $\{A_w:w\in W\}$ is a family of measurable events in a probability
space $(\Omega,\Sigma, \mu)$ satisfying $\mu(A_w)\meg \varepsilon$ for every $w\in W$, then there exists a Carlson--Simpson subtree $U$ of
$W$ with $\dim(U)=m$ and such that for every $u,u'\in U$ we have
\begin{equation} \label{e101}
\mu(A_u\cap A_{u'})\meg \vartheta^2.
\end{equation}
\end{prop}
Proposition \ref{p101} implies, of course, that the constant $\theta(k,\ee,2)$ can be chosen to be $\ee^2-o(1)$. Its proof is given in
\S 10.2 and is based on a detailed analysis of the Ramsey properties of pairs of Carlson--Simpson trees. This analysis is carried out in \S 10.1. 

We mention that, in what follows, for any set $X$ by $\mathrm{Pairs}(X)$ we shall denote the set $\{F\subseteq X: |F|=2\}$. While
unconventional, this notation is quite informative and very convenient.

\subsection{Word representation of pairs of $[k]^{<\nn}$}

Let $k\in\nn$ with $k\meg 2$ and, as in \eqref{e43}, define
\begin{equation} \label{e102}
\mathcal{L}_2 = [k]^2\setminus \big\{(a,a): a\in [k]\big\}.
\end{equation}
Next fix a letter $*$ and let 
\begin{equation} \label{e103}
\mathcal{L}=\mathcal{L}_2 \cup \big\{(a,*): a\in [k]\big\} \cup \big\{(*,a): a\in [k]\big\}.
\end{equation}
Setting $A=[k]\cup\{*\}$, we see that $\mathcal{L}\subseteq A^2$. Also notice that $|\mathcal{L}|=k^2+k$. As in \S 3.1,
by $\mathrm{W}(k,\mathcal{L})$ we denote the set of all words over $[k]\cup\mathcal{L}$.

We shall define an injective map
\begin{equation} \label{e104}
\mathrm{R}: \mathrm{Pairs}([k]^{<\nn})\to \mathrm{W}(k,\mathcal{L})
\end{equation}
as follows. Let $F\in \mathrm{Pairs}([k]^{<\nn})$ be arbitrary. We consider the following cases.
\medskip

\noindent \textsc{Case 1:} \textit{$F$ is flat}. In this case let $\mathrm{R}(F)=\mathrm{R}_2(F)$ where
$\mathrm{R}_2:\mathrm{Fl}_2([k]^{<\nn})\to\mathrm{W}(k,\mathcal{L})$ is the ``representation" map defined in \S 4.3.
\medskip

\noindent \textsc{Case 2:} \textit{$F=\{w,u\}$ with $u$ a successor of $w$}. Write $u$ as $w^{\con}(a_0,...,a_{n-1})$ where $a_i\in [k]$
for every $i\in\{0,...,n-1\}$. We define 
\begin{equation} \label{e105}
\mathrm{R}(F)=w^{\con}(\tilde{a}_0,...,\tilde{a}_{n-1})
\end{equation}
where
\begin{equation} \label{e106}
\tilde{a}_i=\begin{cases}
(*,a_0) & \text{if } a_i=a_0, \\
a_i     & \text{otherwise} \\
\end{cases}
\end{equation}
for every $i\in\{0,...,n-1\}$.
\medskip

\noindent \textsc{Case 3:} \textit{$F=\{w,u\}$ with $w<_{\mathrm{lex}} u$ and $|w|<|u|$}. Let $m=|w|$ and $n=|u|$ and notice that $m<n$.
Write $w$ as $(a_0,...,a_{m-1})$ and $u$ as $(b_0,...,b_{n-1})$, and define 
\begin{equation} \label{e107}
\mathrm{R}(F)=\mathrm{R}_2\big( \{(a_0,...,a_{m-1}),(b_0,...,b_{m-1})\} \big)^{\!\con}\!(\tilde{b}_m,...,\tilde{b}_{n-1})
\end{equation}
where 
\begin{equation} \label{e108}
\tilde{b}_j=\begin{cases}
(*,b_m) & \text{if } b_j=b_m, \\
b_j     & \text{otherwise} \\
\end{cases}
\end{equation}
for every $j\in\{m,...,n-1\}$.
\medskip

\noindent \textsc{Case 4:} \textit{$F=\{w,u\}$ with $w<_{\mathrm{lex}} u$ and $|w|>|u|$}. As in the previous case, let $m=|w|$ and $n=|u|$.
Also write $w$ as $(a_0,...,a_{m-1})$ and $u$ as $(b_0,...,b_{n-1})$. We define 
\begin{equation} \label{e109}
\mathrm{R}(F)=\mathrm{R}_2\big( \{(a_0,...,a_{n-1}),(b_0,...,b_{n-1})\} \big)^{\!\con}\!(\tilde{a}_n,...,\tilde{a}_{m-1})
\end{equation}
where 
\begin{equation} \label{e1010}
\tilde{a}_j=\begin{cases}
(a_n,*) & \text{if } a_j=a_n, \\
a_j     & \text{otherwise} \\
\end{cases}
\end{equation}
for every $j\in\{n,...,m-1\}$.
\medskip

The above cases are exhaustive, and so, this completes the definition of the map $\mathrm{R}$ which is easily seen to be injective.
For every $F\in\mathrm{Pairs}([k]^{<\nn})$ we call the word $\mathrm{R}(F)$ the ``representation" of $F$. We define the
\textit{type} of $F$ to be the type of its representation $\mathrm{R}(F)$ in $\mathrm{W}(k,\mathcal{L})$ (see \S 4.3).

We remark that not all types in $\mathcal{T}(\mathcal{L})$ are realized as the type of a pair of $[k]^{<\nn}$. Actually, the type $\tau$
of a pair $F$ has a very particular form depending, of course, on the nature of $F$. Specifically, we have the following possibilities.
\begin{enumerate}
\item[(1)] In Case 1 the type $\tau$ is of the form $\big((a_0,b_0),...,(a_{\ell-1}, b_{\ell-1})\big)$ where $\ell$ is a positive integer
and $a_0,b_0,...,a_{\ell-1},b_{\ell-1}\in [k]$.
\item[(2)] In Case 2 the type $\tau$ is of the form $\big((*,a)\big)$ where $a\in [k]$. In particular, in this case there are exactly 
$k$ many types. All of them have length $1$.
\item[(3)] In Case 3 the type $\tau$ is of the form $\big((a_0,b_0),...,(a_{\ell-1},b_{\ell-1})\big)^{\!\con}\!(*,b)$ where $\ell$ is
a positive integer and $a_0,b_0,...,a_{\ell-1},b_{\ell-1},b\in [k]$.
\item[(4)] In Case 4 the type $\tau$ is of the form $\big((a_0,b_0),...,(a_{\ell-1},b_{\ell-1})\big)^{\!\con}\!(a,*)$ where $\ell$ is
a positive integer and $a_0,b_0,...,a_{\ell-1},b_{\ell-1},a\in [k]$.
\end{enumerate}

We proceed to discuss two important properties guaranteed by the above definitions. They are isolated in Lemmas \ref{l103} and
\ref{l104} below. However, first we need to introduce some pieces of notation. Specifically, for every Carlson--Simpson tree $W$
of $[k]^{<\nn}$ and every $\tau\in\mathcal{T}(\mathcal{L})$ we set
\begin{equation} \label{e1011}
\mathrm{Pairs}_\tau(W)= \{F\in \mathrm{Pairs}(W): F \text{ is of type } \tau\}.
\end{equation}
Moreover, if $a\in[k]\cup\mathcal{L}$ and $i$ is a positive integer, then let 
\begin{equation} \label{e1012}
a^i=(\underbrace{a,...,a}_{i\mathrm{-times}}).
\end{equation}
By convention, $a^0$ is defined to be the empty word.

The following fact is straightforward.
\begin{fact} \label{f102}
Let $V$ be a Carlson--Simpson tree of $[k]^{<\nn}$ and denote by $\mathbf{v}$ its generating sequence. Then the following hold.
\begin{enumerate}
\item[(i)] The restriction of $\mathrm{R}$ to $\mathrm{Pairs}(V)$ is an injection into $\mathbf{v}(k,\mathcal{L})$.
\item[(ii)] Let $\tau\in\mathcal{T}(\mathcal{L})$ and assume that the set $\mathrm{Pairs}_{\tau}(V)$ is nonempty.
Then the restriction of $\mathrm{R}$ to $\mathrm{Pairs}_\tau(V)$ is onto $\mathbf{v}(k,\mathcal{L},\tau)$. In particular, the map
$\mathrm{R}:\mathrm{Pairs}_\tau(V)\to \mathbf{v}(k,\mathcal{L},\tau)$ is a bijection.
\end{enumerate}
\end{fact}
We are ready to state the first main result of this subsection.
\begin{lem} \label{l103}
Let $k,d,r,n$ be positive integers with $k\meg 2$. Also let $W$ be a Carlson--Simpson tree of $[k]^{<\nn}$ and assume that
\begin{equation} \label{e1013}
\dim(W)\meg \mathrm{FK}(k,k^2+k,d,r).
\end{equation}
Then for every $r$-coloring of the set $\mathrm{Pairs}(W)$ there exists a Carlson--Simpson subtree $V$ of $W$ with $\dim(V)=d$
such that for every type $\tau\in\mathcal{T}(\mathcal{L})$ the set $\mathrm{Pairs}_{\tau}(V)$ is either empty or monochromatic. 
\end{lem}
\begin{proof} 
Fix a coloring $c:\mathrm{Pair}(W)\to [r]$ and define $\tilde{c}:\mathbf{w}(k,\mathcal{L})\to [r]$ by the rule
\begin{equation} \label{e1014}
\tilde{c}(w)=
\begin{cases}
c(F)  & \text{if there exists } F\in\mathrm{Pairs}(W) \text{ with } \mathrm{R}(F)=w, \\
r    & \text{otherwise}. \\
\end{cases}
\end{equation}
By Fact \ref{f102}, the coloring $\tilde{c}$ is well-defined. Next recall that $|\mathcal{L}|=k^2+k$. Therefore, by \eqref{e1013}
and Theorem \ref{t14}, there exists a $d$-dimensional Carlson--Simpson subsequence $\mathbf{v}$ of $\mathbf{w}$ such that for every
$\tau\in\mathcal{T}(\mathcal{L})$ the set $\mathbf{v}(k,\mathcal{L},\tau)$ is monochromatic. Let $V$ be the Carlson--Simpson subtree
of $W$ generated by $\mathbf{v}$. Invoking Fact \ref{f102} we see that $V$ is as desired. 
\end{proof}
The final result of this subsection is the following lemma.
\begin{lem} \label{l104}
Let $k,m\in\nn$ with $k\meg 2$ and $m\in\nn$, and $0<\vartheta<\varepsilon\mik 1$. Also let $V$ be a Carlson--Simpson tree of $[k]^{<\nn}$
with 
\begin{equation} \label{e1015}
\dim(V)\meg \lceil (\varepsilon^2-\vartheta^2)^{-1}\rceil \cdot (m+1).
\end{equation}
Finally let $\{A_v:v\in V\}$ be a family of measurable events in a probability space $(\Omega,\Sigma,\mu)$ satisfying $\mu(A_v)\meg\ee$
for every $v\in V$. If $\tau\in \mathcal{T}(\mathcal{L})$ is a type such that $\mathrm{Pairs}_{\tau}(V)\neq\varnothing$ and $|\tau|\mik m$,
then there exists $\{v,v'\}\in \mathrm{Pairs}_{\tau}(V)$ such that
\begin{equation} \label{e1016}
\mu(A_v\cap A_{v'})\meg\vartheta^2.
\end{equation}
\end{lem}
\begin{proof} 
We set $d=\dim(V)$. Clearly we may assume that $V=[k]^{<d+1}$. Also let
\begin{equation} \label{e1017}
n_0=\lceil(\varepsilon^2-\vartheta^2)^{-1}\rceil.
\end{equation}
Fix a type $\tau\in\mathcal{T}(\mathcal{L})$ with $|\tau|\mik m$ and such that $\mathrm{Pairs}_{\tau}([k]^{<d+1})\neq\varnothing$. We need to find
$\{v,v'\}\in\mathrm{Pairs}_{\tau}([k]^{<d+1})$ such that the estimate in \eqref{e1016} is satisfied for the pair $\{v,v'\}$. The argument below
is not uniform and depends on the form of the type $\tau$. In particular, we distinguish the following cases. 
\medskip

\noindent \textsc{Case 1:} \textit{we have that $\tau=\big((a_0,b_0),...,(a_{\ell-1}, b_{\ell-1})\big)$ where $|\tau|=\ell\in [m]$
and $a_0,b_0,...,a_{\ell-1},b_{\ell-1}\in [k]$}. For every $i\in\{0,...,n_0\}$ define
\begin{equation} \label{e1018}
v_i=(b_0^i)^{\con}(a_0^{n_0-i})^{\con}(b_1^i)^{\con}(a_1^{n_0-i})^\con...^\con(b_{\ell-1}^i)^{\con}(a_{\ell-1}^{n_0-i}).
\end{equation}
Notice that $v_i\in [k]^{n_0\cdot\ell}$. Hence, by \eqref{e1015}, we have $v_i\in [k]^{<d+1}$ which implies that $\mu(A_{v_i})\meg\ee$.
By \eqref{e1017} and Lemma \ref{l27}, there exist $i,j\in\{0,...,n_0\}$ with $i<j$ such that $\mu(A_{v_i}\cap A_{v_j})\meg \vartheta^2$.
We set $F=\{v_i,v_j\}$. It is enough to show that $F\in\mathrm{Pairs}_{\tau}([k]^{<d+1})$. Indeed, by \eqref{e1018}, we see that $\mathrm{R}(F)$
is the word
\begin{equation} \label{e1019}
(b_0^i)^{\con}\big((a_0, b_0)^{j-i}\big)^{\con}(a_0^{n_0-j})^\con ...^\con(b_{\ell-1}^i)^\con 
\big((a_{\ell-1}, b_{\ell-1})^{j-i}\big)^\con (a_{\ell-1}^{n_0-j}).
\end{equation}
This easily implies that the type of $F$ is $\tau$. Since $F\subseteq [k]^{<d+1}$ we conclude that $F\in\mathrm{Pairs}_{\tau}([k]^{<d+1})$.
\medskip

\noindent \textsc{Case 2:} \textit{we have that $\tau=\big((*,a)\big)$ where $a\in [k]$}. For every $i\in\{0,...,n_0\}$ let
$v_i=a^i$ and notice that $v_i\in [k]^i\subseteq [k]^{<d+1}$. In particular, we have $\mu(A_{v_i})\meg \ee$. Arguing precisely
as in the previous case, we find $i,j\in\{0,...,n_0\}$ with $i<j$ and such that $\mu(A_{v_i}\cap A_{v_j})\meg \vartheta^2$. It is then easy
to see that $F=\{v_i,v_j\}$ is as desired. 
\medskip

\noindent \textsc{Case 3:} \textit{we have that $\tau=\big((a_0,b_0),...,(a_{\ell-1},b_{\ell-1})\big)^\con(*,b)$ where $|\tau|-1=\ell\in [m-1]$
and $a_0,b_0,..., a_{\ell-1}, b_{\ell-1},b\in [k]$}. For every $i\in\{0,...,n_0\}$ define
\begin{equation} \label{e1020}
v_i=(b_0^i)^{\con}(a_0^{n_0-i})^{\con}(b_1^i)^{\con}(a_1^{n_0-i})^\con...^\con(b_{\ell-1}^i)^{\con}(a_{\ell-1}^{n_0-i})^\con (b^i)
\end{equation}
and notice that $v_i\in [k]^{n_0\cdot \ell+i}\subseteq [k]^{<d+1}$. Arguing as above, we find $i,j\in\{0,...,n_0\}$ with $i<j$ and such that
$\mu(A_{v_i}\cap A_{v_j})\meg \vartheta^2$. We set $F=\{v_i,v_j\}$. Clearly it is enough to show that $F\in\mathrm{Pairs}_{\tau}([k]^{<d+1})$.
To this end we observe that, by \eqref{e1020}, the ``representation" $\mathrm{R}(F)$ of $F$ is the word
\begin{eqnarray} \label{e1021}
& & (b_0^i)^{\con}\big((a_0, b_0)^{j-i}\big)^{\con}(a_0^{n_0-j})^\con ... \\
& & \ \ \ \ \ \ \ \ \ \ \ \ ...^\con(b_{\ell-1}^i)^\con \big((a_{\ell-1}, b_{\ell-1})^{j-i}\big)^\con
(a_{\ell-1}^{n_0-j})^{\con}(b^i)^{\con}\big((*,b)^{j-i}\big). \nonumber
\end{eqnarray}
It follows that the type of $F$ is $\tau$ and, in particular, that $F\in\mathrm{Pairs}_{\tau}([k]^{<d+1})$.
\medskip

\noindent \textsc{Case 4:} \textit{we have that $\tau=\big((a_0,b_0),...,(a_{\ell-1},b_{\ell-1})\big)^\con(a,*)$ where $|\tau|-1=\ell\in\! [m-1]$
and $a_0,b_0,..., a_{\ell-1}, b_{\ell-1},a\in [k]$}. For every $i\in\{0,...,n_0\}$ let
\begin{equation} \label{e1022}
v_i=(b_0^i)^{\con}(a_0^{n_0-i})^{\con}(b_1^i)^{\con}(a_1^{n_0-i})^\con...^\con(b_{\ell-1}^i)^{\con}(a_{\ell-1}^{n_0-i})^\con (a^i)
\end{equation}
and notice that $v_i\in [k]^{n_0\cdot \ell+i}\subseteq [k]^{<d+1}$. Hence, there exist $0\mik i<j\mik n_0$ with 
$\mu(A_{v_i}\cap A_{v_j})\meg \vartheta^2$ and such that, setting $F=\{v_i,v_j\}$, the ``representation" $\mathrm{R}(F)$ of $F$ is the word
\begin{eqnarray} \label{e1023}
& & (b_0^i)^{\con}\big((a_0, b_0)^{j-i}\big)^{\con}(a_0^{n_0-j})^\con ... \\
& & \ \ \ \ \ \ \ \ \ \ \ \ ...^\con(b_{\ell-1}^i)^\con \big((a_{\ell-1}, b_{\ell-1})^{j-i}\big)^\con
(a_{\ell-1}^{n_0-j})^{\con}(a^i)^{\con}\big((a,*)^{j-i}\big). \nonumber
\end{eqnarray}
Therefore, the pair $F$ is as desired. The above cases are exhaustive and so the entire proof is completed.
\end{proof}

\subsection{Proof of Proposition \ref{p101}}

We set 
\begin{equation} \label{e1024}
\mathrm{Cor}_{2}(k,m,\vartheta,\varepsilon)=  \mathrm{FK}\big(k,k^2+k,\lceil (\varepsilon^2-\vartheta^2)^{-1}\rceil \cdot (m+1),2\big)
\end{equation} 
and we claim that with this choice the result follows.

Indeed, fix a Carlson--Simpson tree $W$ of $[k]^{<\nn}$ with $\dim(W)\meg \mathrm{Cor}_2(k,m,\vartheta,\ee)$ and a family $\{A_w:w\in W\}$
of measurable events in a probability space $(\Omega,\Sigma,\mu)$ satisfying $\mu(A_w)\meg\ee$ for every $w\in W$. We set 
\begin{equation} \label{e1025}
\mathcal{P}=\Big\{ F\in \mathrm{Pairs}(W): \mu\Big( \bigcap_{w\in F}A_w\Big)\meg \vartheta^2\Big\}.
\end{equation}
By \eqref{e1024} and Lemma \ref{l103}, there exists a Carlson-Simpson subtree $V$ of $W$ with $\dim(V)=\lceil (\varepsilon^2-\vartheta^2)^{-1}
\rceil\cdot (m+1)$ and such that for every type $\tau\in\mathcal{T}(\mathcal{L})$ with $\mathrm{Pairs}_{\tau}(V)\neq\varnothing$ we have that
either $\mathrm{Pairs}_\tau(V)\subseteq \mathcal{P}$ or $\mathrm{Pairs}_\tau(V)\cap\mathcal{P}=\varnothing$. By the estimate on the dimension
of $V$ and Lemma \ref{l104}, we see that $\mathrm{Pairs}_\tau(V)\cap\mathcal{P}\neq\varnothing$ provided that $\mathrm{Pairs}_\tau(V)$ is
nonempty and $|\tau|\mik m$. Therefore, for every type $\tau$ with $|\tau|\mik m$ we have $\mathrm{Pairs}_\tau(V)\subseteq \mathcal{P}$.
Let $U$ be any $m$-dimensional Carlson--Simpson subtree of $V$. Clearly $U$ is as desired. 


\section{Free sets}

\numberwithin{equation}{section}

In this section we discuss a second quantitative refinement of Theorem \ref{t12}. Specifically our goal is to obtain optimal lower bounds for
the correlation of the events $\{A_t: t\in F\}$ in \eqref{e13} provided that the set $F$ is \textit{free}. The class of free sets was introduced
in \cite{DKT1} and includes various well-known subsets of Carlson--Simpson trees such as all finite chains, all doubletons and many more.
Its definition and main properties are recalled in \S 11.1. In \S 11.2 we state the main result of this section, Theorem \ref{t114} below.
The proof of Theorem \ref{t114} is given in \S 11.3. 

\subsection{Free sets: definitions and main properties}

We start with the following definition (see \cite[Definition 6.1]{DKT1}).
\begin{defn} \label{d111}
Let $k\in\nn$ with $k\meg 2$. Recursively, for every integer $n\meg 1$ we define a family $\mathrm{Fr}_n([k]^{<\nn})$ as follows. First let
$\mathrm{Fr}_1([k]^{<\nn})$ and $\mathrm{Fr}_2([k]^{<\nn})$ consist of all singletons and all doubletons of $[k]^{<\nn}$ respectively. Let
$n\in\nn$ with $n\meg 2$ and assume that the family $\mathrm{Fr}_n([k]^{<\nn})$ has been defined. Then $\mathrm{Fr}_{n+1}([k]^{<\nn})$ consists
of all subsets of $[k]^{<\nn}$ which can be written in the form $\{w\}\cup G$ where $w\in [k]^{<\nn}$ and $G\in \mathrm{Fr}_n([k]^{<\nn})$ are
such that $|w|<|\!\wedge G|$. We set
\begin{equation} \label{e111}
\mathrm{Fr}([k]^{<\nn})=\bigcup_{n\meg 1}\mathrm{Fr}_n([k]^{<\nn}).
\end{equation}
An element of $\mathrm{Fr}([k]^{<\nn})$ will be called a \emph{free} subset of $[k]^{<\nn}$.
\end{defn}
We have the following characterization of free sets. Its proof if straightforward.
\begin{fact} \label{f112}
Let $k\in\nn$ with $k\meg 2$. Also let $F$ be a finite subset of $[k]^{<\nn}$ with $|F|\meg 3$. Then $F$ is free if and only if there exists
an enumeration $\{w_1,...,w_n\}$ of $F$ such that
\begin{enumerate}
\item[(1)] $|w_1|< |w_2|< \cdots <|w_{n-1}|\mik |w_n|$ and
\item[(2)] $|w_i|<|\!\wedge\{w_{i+1},...,w_n\}|$ for every $i\in\{1,...,n-2\}$.
\end{enumerate}
\end{fact}
For every Carlson--Simpson tree $W$ of $[k]^{<\nn}$ by $\mathrm{Fr}(W)$ we shall denote the set of all free subsets of $[k]^{<\nn}$ which
are contained in $W$. Moreover, for every integer $n\meg 1$ we set
\begin{equation} \label{e112}
\mathrm{Fr}_n(W)=\{ F\in \mathrm{Fr}(W): |F|=n\}.
\end{equation}
We isolate, for future use, the following fact. Its proof is also straightforward. 
\begin{fact} \label{f113}
Let $k,d\in\nn$ with $k\meg 2$ and $d\meg 1$. Also let $W$ be a $d$-dimensional Carlson--Simpson tree of $[k]^{<\nn}$.  
Then the following hold.
\begin{enumerate}
\item[(a)] For every $F\subseteq [k]^{<d+1}$ we have that $F$ is free if and only if $\mathrm{I}_W(F)$ is free, where $\mathrm{I}_W$
is the canonical isomorphism associated to $W$.
\item[(b)] For every integer $n\meg 1$ we have that $\mathrm{Fr}_n(W)\neq\varnothing$ if and only if $n-1\mik d$.
\item[(c)] For every Carlson--Simpson subtree $V$ of $W$ we have that $\mathrm{Fr}_n(V)\subseteq \mathrm{Fr}_n(W)$.
\end{enumerate}
\end{fact}

\subsection{The main result} 

We have the following extension of Proposition \ref{p101}.
\begin{thm} \label{t114} 
Let $k,m,n\in\nn$ with $k\meg 2$ and $m+1\meg n\meg 1$, and $0<\vartheta<\varepsilon\mik 1$. Then there exists a positive integer
$\mathrm{Cor}_{n}(k,m, \vartheta,\varepsilon)$ with the following property. If $W$ is a Carlson--Simpson tree of $[k]^{<\nn}$ of dimension
at least $\mathrm{Cor}_{n}(k,m, \vartheta,\varepsilon)$, then for every family $\{A_w: w\in W\}$ of measurable events in a probability
space $(\Omega,\Sigma,\mu)$ satisfying $\mu(A_w)\meg \varepsilon$ for every $w\in W$ there exists a Carlson--Simpson subtree $U$ of $W$
with $\dim(U)=m$ and such that for every $F\in \mathrm{Fr}_n(U)$ we have
\begin{equation} \label{e113}
\mu\Big( \bigcap_{w\in F} A_w \Big)\meg \vartheta^n.
\end{equation}
\end{thm}
By iterating Theorem \ref{t114}, we can control simultaneously all free subsets of a Carlson--Simpson tree. Specifically, for every $k,m\in\nn$
with $k\meg 2$ and $m\meg 1$, and every $0<\vartheta<\varepsilon\mik 1$ we define a sequence $(N_i)_{i=1}^{m+1}$ in $\nn$ recursively by the rule
\begin{equation} \label{e114}
\begin{cases} N_1=m, \\ N_{i+1}=\mathrm{Cor}_{i+1}(k,N_i, \vartheta,\varepsilon) \end{cases} 
\end{equation}
and we set 
\begin{equation} \label{e115}
\mathrm{Cor}_{\mathrm{Fr}}(k,m,\vartheta,\ee)=N_{m+1}.
\end{equation}
We have the following corollary. 
\begin{cor} \label{c114} 
Let $k,m\in\nn$ with $k\meg 2$ and $m\meg 1$, and $0<\vartheta<\varepsilon\mik 1$. If $W$ is a Carlson--Simpson tree of $[k]^{<\nn}$ of
dimension at least $\mathrm{Cor}_{\mathrm{Fr}}(k,m, \vartheta,\varepsilon)$, then for every family $\{A_w: w\in W\}$ of measurable events
in a probability space $(\Omega,\Sigma,\mu)$ satisfying $\mu(A_w)\meg \varepsilon$ for every $w\in W$ there exists a Carlson--Simpson subtree
$U$ of $W$ with $\dim(U)=m$ and such that for every $F\in \mathrm{Fr}(U)$ we have
\begin{equation} \label{e116}
\mu\Big( \bigcap_{w\in F} A_w \Big)\meg \vartheta^{|F|}.
\end{equation}
\end{cor}

\subsection{Proof of Theorem \ref{t114}}

For every probability space $(\Omega,\Sigma,\mu)$, every $Y\in\Sigma$ with $\mu(Y)>0$ and every $A\in\Sigma$ let
\begin{equation} \label{e117}
\mu(A \ | \ Y)=\frac{\mu(A\cap Y)}{\mu(Y)}
\end{equation}
be the conditional probability of $A$ relative to $Y$. We have the following proposition.
\begin{prop} \label{p116}
Let $k,d\in\nn$ with $k\meg 2$ and $d\meg 1$, and $0<\eta<\varrho\mik 1$. Then there exists a positive integer $\mathrm{Rel}(k,d,\eta,\varrho)$
with the following property. If $W$ is a Carlson--Simpson tree $W$ of $[k]^{<\nn}$  of dimension at least $\mathrm{Rel}(k,d,\eta,\varrho)$, then
for every family $\{A_w:w\in W\}$ of measurable events in a probability space $(\Omega,\Sigma,\mu)$ satisfying $\mu(A_w)\meg\varrho$ for every
$w\in W$ there exists a Carlson--Simpson subtree $V$ of $W$ with $\dim(V)=d$ and such that for every $w,v\in V$ we have
\begin{equation} \label{e118}
\mu(A_w \ | \ A_v)\meg\eta.
\end{equation}
\end{prop}
\begin{proof}
First we introduce some numerical invariants. We set
\begin{equation} \label{e119}
\lambda=\big( \varrho\cdot \eta^{-1}\big)^{\frac{1}{3}}.
\end{equation}
Notice that $\lambda>1$. Also let
\begin{equation} \label{e1110}
r=\Big\lceil\frac{\log\varrho^{-1}}{\log\lambda}\Big\rceil.
\end{equation}
It is easy to see that
\begin{equation} \label{e1111}
0<\eta<\lambda^{-1}\varrho < \varrho < \lambda \varrho <\lambda^2\varrho < \cdots < \lambda^{r-1}\varrho< 1\mik \lambda^{r}\varrho.
\end{equation}
For every $i\in\{0,...,r-1\}$ let
\begin{equation} \label{e1112}
d_i=\mathrm{Cor}_{2}(k, d, \lambda^{i-1}\varrho,\lambda^{i}\varrho)
\end{equation}
and define
\begin{equation} \label{e1113}
\Delta=\max\{d_i: 0\mik i\mik r-1\}.
\end{equation}
We set
\begin{equation} \label{e1114}
\mathrm{Rel}(k,d,\eta,\varrho)= \mathrm{CS}(k,\Delta+1,1,r)
\end{equation}
and we claim that with this choice the result follows. 

Indeed let $W$ be a Carlson--Simpson tree of $[k]^{<\nn}$ with $\dim(W)\meg \mathrm{Rel}(k,d,\eta,\varrho)$ and $\{A_w:w\in W\}$ be a family
of measurable events in a probability space $(\Omega,\Sigma,\mu)$ satisfying $\mu(A_w)\meg \varrho$ for every $w\in W$. By \eqref{e1114} and
Lemma \ref{l26}, there exist a Carlson--Simpson subtree $U$ of $W$ with $\dim(U)=\Delta$ and $i_0\in\{0,...,r-1\}$ such that for every
$u\in U$ we have 
\begin{equation} \label{e1115}
\lambda^{i_0}\varrho\mik \mu(A_u)\mik \lambda^{i_0+1}\varrho.
\end{equation}
Notice that 
\begin{equation} \label{e1116}
\dim(U) \meg \mathrm{Cor}_{2}(k, d, \lambda^{i_0-1}\varrho,\lambda^{i_0}\varrho).
\end{equation}
Therefore, by Proposition \ref{p101}, there exists a Carlson--Simpson subtree $V$ of $U$ with $\dim(V)=d$ and such that
\begin{equation} \label{e1117}
\mu(A_v\cap A_{v'})\meg \lambda^{2i_0-2}\varrho^2
\end{equation}
for every $v,v'\in V$. By \eqref{e1115} and \eqref{e1117}, and taking into account that $\lambda>1$ and $i_0\meg 0$, we conclude that
\begin{equation} \label{e1118}
\mu(A_v \ | \ A_{v'})= \frac{\mu(A_v\cap  A_{v'})}{\mu( A_{v'})} \meg \frac{\lambda^{2i_0-2}\varrho^2 }{\lambda^{i_0+1}\varrho}
= \lambda^{i_0-3}\varrho \meg \lambda^{-3} \varrho\stackrel{\eqref{e119}}{=} \eta.
\end{equation}
The proof is completed.
\end{proof}
We are ready to proceed to the proof of Theorem \ref{t114}.
\begin{proof}[Proof of Theorem \ref{t114}]
The proof proceeds by induction on $n$. Of course the case ``$n=1$" is straightforward, while the case ``$n=2$" follows from Proposition
\ref{p101}. So let $n\in\nn$ with $n\meg 2$ and assume that the numbers $\mathrm{Cor}_n(k,m,\vartheta,\ee)$ have been defined for any
choice of admissible parameters. We fix $k,m\in\nn$ with $k\meg 2$ and $m\meg n$, and $0<\vartheta<\ee\mik 1$. To define the number
$\mathrm{Cor}_{n+1}(k,m,\vartheta,\ee)$ we need to do some preparatory work. First let $f:\nn\to\nn$ be defined by the rule
\begin{equation} \label{e1119}
f(x)=\mathrm{Cor}_{n}\Big(k,x,\vartheta, \frac{\vartheta+\varepsilon}{2}\Big)
\end{equation}
if $x\meg n-1$ and $f(x)=0$ otherwise. Next we define a sequence $(d_i)_{i=0}^{m-n+1}$ in $\nn$ recursively by the rule
\begin{equation} \label{e1120}
\begin{cases} d_{m-n+1}=n-1, \\ d_{i-1}= 1+f^{(k^{2i-1})}(d_i) \end{cases}
\end{equation}
where $f^{(k^{2i-1})}$ stands for the $(k^{2i-1})$-th iteration of $f$. Finally let
\begin{equation} \label{e1121}
\mathrm{Cor}_{n+1}(k,m, \vartheta,\varepsilon)= \mathrm{Rel}\Big(k,d_0,\frac{\vartheta+\varepsilon}{2},\varepsilon\Big).
\end{equation}
We claim that with this choice the result follows.

Indeed, fix a Carlson--Simpson tree $W$ of $[k]^{<\nn}$ with
\begin{equation} \label{e1122}
\dim(W)\meg \mathrm{Rel}\Big(k,d_0,\frac{\vartheta+\varepsilon}{2},\varepsilon\Big).
\end{equation}
Also let $\{A_w: w\in W\}$ be a family of measurable events in a probability space $(\Omega,\Sigma,\mu)$ satisfying $\mu(A_w)\meg\ee$ for every
$w\in W$. By \eqref{e1122} and Proposition \ref{p116}, there exists a Carlson--Simpson subtree $V$ of $W$ with $\dim(V)=d_0$ and such that
\begin{equation} \label{e1123}
\mu(A_v \ | \ A_{v'})\meg \frac{\vartheta+\varepsilon}{2}
\end{equation}
for every $v,v'\in V$. 

Recursively we shall construct a sequence $(V_i)_{i=0}^{m-n+1}$ of Carlson--Simpson trees of $[k]^{<\nn}$ with $V_0=V$ and such that
for every $i\in[m-n+1]$ the following conditions are satisfied.
\begin{enumerate}
\item[(C1)] We have that $V_i$ is a Carlson--Simpson subtree of $V_{i-1}$ of dimension $i+d_i$.
\item[(C2)] For every $j\in\{0,...,i-1\}$ we have that $V_i(j)=V_{i-1}(j)$.
\item[(C3)] For every $v\in V_i(i-1)$ and every $G\in\mathrm{Fr}_n(V_i)$ with $\wedge G\in V_i(j)$ for some $j\in \{i,...,\dim(V_i)-1\}$
we have
\begin{equation} \label{e1124}
\mu\Big( \bigcap_{w\in G} A_w \ | \ A_v\Big)\meg \vartheta^{n}.
\end{equation}
\end{enumerate}
We proceed to the construction. Let $i\in [m-n+1]$ and assume that the Carlson--Simpson trees $V_0,...,V_{i-1}$ have been selected so that
conditions (C1), (C2) and (C3) are satisfied. Set 
\begin{equation} \label{e1125}
d=\dim(V_{i-1})\stackrel{(\mathrm{C1})}{=} i-1+d_{i-1}
\end{equation} 
and denote by $\mathbf{v}=(v,v_0,...,v_{d-1})$ the generating sequence of $V_{i-1}$. Let $T$ be the Carlson--Simpson tree generated by the
Carlson--Simpson sequence $(\varnothing, v_{i},..., v_{d-1})$. Notice that
\begin{equation} \label{e1126}
\dim(T)= d-i\stackrel{\eqref{e1125}}{=}d_{i-1}-1\stackrel{\eqref{e1120}}{=}f^{(k^{2i-1})}(d_i).
\end{equation}
Also observe that $|V_{i-1}(i-1)|=k^{i-1}$ and $|V_{i-1}(i)|=k^i$. By \eqref{e1123}, the definition of the map $f$ in \eqref{e1119} and using
our inductive assumptions successively $k^{2i-1}$ many times, we find a Carlson--Simpson subtree $S$ of $T$ with $\dim(S)=d_i$ and such that
for every $v\in V_{i-1}(i-1)$, every $t\in V_{i-1}(i)$ and every $G\in \mathrm{Fr}_{n}( t^\con S)$ we have 
\begin{equation} \label{e1127}
\mu \Big( \bigcap_{w\in G} A_w \ | \ A_{v}\Big) \meg \vartheta^n.
\end{equation}
We set
\begin{equation} \label{e1128}
V_i=\Big( \bigcup_{j=0}^{i-1} V_{i-1}(j)\Big) \cup \Big(\bigcup_{t\in V_{i-1}(i)} t^\con S\Big).
\end{equation}
It is easily checked that $V_i$ satisfies conditions (C1)--(C3). The recursive selection is thus completed.
 
We are ready for the final step of the argument. We set $U=V_{m-n+1}$ and we observe that
\begin{equation} \label{e1129}
\dim(U)\stackrel{\mathrm{(C1)}}{=}m-n+1+d_{m-n+1}\stackrel{\eqref{e1120}}{=} m.
\end{equation}
Let $H\in\mathrm{Fr}_{n+1}(U)$ be arbitrary. Write $H$ as $\{v\}\cup G$ where $v\in U$ and $G\in \mathrm{Fr}_n(U)$ are such that
$|v|<|\!\wedge G|$. By Fact \ref{f113} and \eqref{e1129}, we see that there exists $i_0\in \{0,...,m-n+1\}$ such that $v\in U(i_0-1)$. 
Invoking conditions (C1) and (C2), we get that $v\in V_{i_0}(i_0-1)$, $G\in\mathrm{Fr}_n(V_{i_0})$ and $\wedge G\in V_{i_0}(j)$ 
for some $j\meg i_0$. Hence, by (C3), we conclude that
\begin{equation} \label{e1130}
\mu \Big( \bigcap_{w\in H} A_w \Big)= \mu\Big( \bigcap_{w\in G} A_{w} \ | \ A_v\Big) \cdot \mu (A_v) \stackrel{\eqref{e1124}}{\meg}
\vartheta^{n}\cdot \varepsilon >\vartheta^{n+1}.
\end{equation}
This completes the proof of the general inductive step and so the entire proof is completed.
\end{proof}


\end{document}